\documentclass[a4paper,11pt,english]{amsart}
\usepackage[a4paper]{geometry}
\usepackage[utf8]{inputenc}
\usepackage{cite}
\usepackage{url}
\usepackage{xcolor}
\usepackage{graphicx}
\usepackage{soul}
\usepackage[normalem]{ulem}
\usepackage{mathtools}  
\mathtoolsset{showonlyrefs}
\usepackage{hyperref}

\usepackage[foot]{amsaddr}
\usepackage{enumerate}

\usepackage{amssymb}
\allowdisplaybreaks[4]
\usepackage{mathscinet}

\usepackage{bbm}

\newtheorem{theorem}{Theorem}[section]
\newtheorem{lemma}[theorem]{Lemma}
\newtheorem{definition}[theorem]{Definition}

\newtheorem{proposition}[theorem]{Proposition}
\newtheorem{remark}[theorem]{Remark}

\newtheorem{hypothesis}{Hypothesis}

\newtheorem{fact}[theorem]{Fact}

\numberwithin{equation}{section}

\newcommand{\Bcal} {{\mathcal B}}

\newcommand{\Fcal} {{\mathcal F}}

\newcommand{\Mcal} {{\mathcal M}}

\newcommand{\Ocal} {{\mathcal O}}

\newcommand{\Z}{\mathbb{Z}}
\newcommand{\R}{\mathbb{R}}
\newcommand{\N}{\mathbb{N}}

\renewcommand{\P}{\mathbb{P}}
\newcommand{\E}{\mathbb{E}}

\newcommand{\T}{\mathbb{T}}

\renewcommand{\epsilon}{\varepsilon}
\newcommand{\eps}{\varepsilon} 

\newcommand{\Lip}{\operatorname{Lip}}

\newcommand{\<}{\langle}
\renewcommand{\>}{\rangle}
\newcommand{\ud}{\mathrm{d}}
\newcommand{\e}{\varepsilon}

\usepackage{graphicx}
\usepackage{tikz}
\usepackage{atbegshi}

\AtBeginShipoutFirst{%
    \begin{tikzpicture}[remember picture, overlay]
        \node[anchor=north west, xshift=1in, yshift=-1.5cm] at (current page.north west) {%
            \includegraphics[width=2cm]{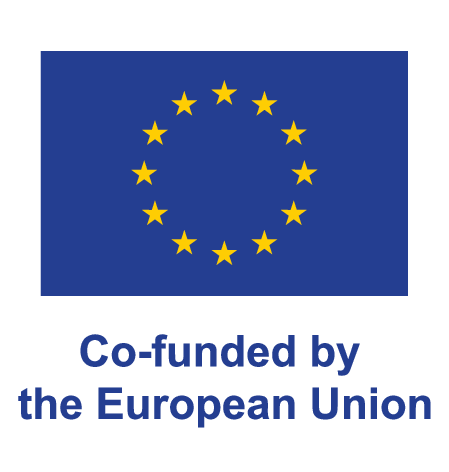} 
        };
    \end{tikzpicture}%
}

\title[Ergodicity and mixing for stochastic evolution equations]{Ergodicity and mixing for locally monotone stochastic evolution equations}
\author[G. Barrera]{Gerardo Barrera}
\email{gerardo.barrera.vargas@tecnico.ulisboa.pt}

\address[GB]{Center for Mathematical Analysis\\ Geometry and Dynamical Systems\\ Instituto
Superior T\'ecnico\\ Universidade de Lisboa\\ Av. Rovisco Pais\\ 1049-001 Lisboa\\
Portugal}

\author[J. M. T\"olle]{Jonas M. T\"olle}
\email{jonas.tolle@aalto.fi}

\address[JMT]{Aalto University\\
Department of Mathematics and Systems Analysis\\
PO Box 11100 (Otakaari 1, Espoo)\\
00076 Aalto\\
Finland}

\date{\today}

\keywords{Additive Gaussian noise; 
exponentially ergodic Markovian Feller semigroup; $e$-property;
locally monotone drift stochastic evolution equation with degenerate Wiener noise;
stochastic Burgers equation;
stochastic incompressible 2D Navier-Stokes equations;
stochastic power law fluids;
stochastic semilinear equations;
unique invariant probability measure; mixing times}
\subjclass{35K65; 35R60; 37A25; 37L40; 47D07; 60H15; 76D05}
\begin{document}
\begin{abstract}
We establish general quantitative conditions for stochastic evolution equations with locally monotone drift and degenerate additive Wiener noise in variational formulation resulting in the existence of a unique invariant probability measure for the associated exponentially ergodic Markovian Feller semigroup. We prove improved moment estimates for the solutions and the $e$-property of the semigroup.
Furthermore, we provide quantitative upper bounds for the $2$-Wasserstein $\varepsilon$-mixing times.
Examples on possibly unbounded domains include the stochastic incompressible 2D Navier-Stokes equations, shear thickening stochastic power-law fluid equations, the stochastic heat equation, as well as, stochastic semilinear equations such as the 1D stochastic Burgers equation.
\end{abstract}
\maketitle

{\footnotesize
\tableofcontents
}

\section{Introduction}

In this work, we study ergodicity and quantitative upper bounds of mixing times for the Markovian dynamics associated to a general class of stochastic partial differential equations (SPDEs) on a separable Hilbert space, that is, locally monotone drift stochastic evolution equations with spatially degenerate additive Wiener noise with possibly infinitely many modes. Even if the degeneracy of the noise could be seen as a drawback, it is actually this contribution's strength and main novelty. The main reason is the absence of robust techniques for proving ergodicity for those SPDEs with degenerate noise that are not strongly dissipative and thus exponentially ergodic as in~\cite{BDP:06,GM:01}.

For non-degenerate noise, more precisely, for noise with minimal non-degeneracy assumptions on the Fourier modes of the spatial structure, the situation is entirely different. We would like to point out that this kind of minimal type of non-degeneracy is sometimes called ``degenerate'' (as opposed to space-time white noise) or ``mildly degenerate'' in the literature. In this work, we reserve the terminology ``degenerate'' for noise which may be zero, finite-dimensional, or spatially regular. Following this terminology, for non-degenerate noise, the main approach derives from the notion of the strong Feller property of the semigroup, combined with irreducibility, which implies uniqueness of the invariant probability measure and therefore ergodicity of the semigroup. This technique works well for both additive and multiplicative noises~\cite{DPZ:96}. Flandoli and Maslowski used this method to obtain the ergodicity for the stochastic incompressible 2D Navier-Stokes equations with additive non-degenerate Gaussian forcing~\cite{FM:95}. Their non-degeneracy assumption was removed by Mattingly \cite{M:02} and Bricmont, Kupiainen and Lefevere \cite{BKL:02} for large viscosity. The result has been further refined by Hairer and Mattingly toward the notion of the asymptotic strong Feller property \cite{HM:11,HM:06}, which they used to prove the existence and uniqueness of an invariant measure for the
stochastic 2D Navier-Stokes equations with additive Gaussian forcing that admits at least four independent Fourier modes on the two-dimensional torus~\cite{HM:06}. We note that their method yields convergence of the ergodic semigroup in total variation distance, whereas the method used by us yields \emph{weak$^\ast$-mean ergodicity}, where the time-average of the dual semigroup converges weakly in the sense of probability measures. Our result relies on the ergodicity result of Komorowski, Peszat and Szarek~\cite{KPS:10}, which is a refinement of the lower bound technique of Lasota and Szarek~\cite{LS:06}, where the main ingredient is the \emph{$e$-property} of the semigroup, a type of uniform equicontinuity on bounded Lipschitz functions, combined with an asymptotic uniform lower bound for the time-average of the semigroup, which does neither require tightness, as the well-known Krylov-Bogoliubov method \cite{DPZ:96}, nor the sequential weak Feller property, as the method of Maslowski and Seidler \cite{MS:99}. This technique is related to the alternative approach in \cite{BCR:19}. This allows us to consider a general setup which dispenses with a compact embedding of the energy space. This enables us to consider ergodicity for SPDEs on certain unbounded domains, namely, those still satisfying a Poincar\'e inequality, also known as \emph{Poincar\'e domains}, where the Sobolev embedding is not necessarily compact. At the moment, we are only aware of the work \cite{KC:25,K:26}, which treats the existence of invariant measures for unbounded domains for models comparable to ours. At the borderline, our method does not yield mixing. However, we may pose quantitative conditions which yield exponential mixing and exponential ergodicity. This corresponds to the case of large viscosity relative to the noise intensity for the 2D Navier-Stokes equations. To the best of our knowledge, the possible degeneracy of the noise and the explicit bound for the mixing times are novel.

For strongly dissipative monotone drift SPDEs with degenerate Gaussian noise, ergodicity has been obtained, among many other contributions, in~\cite{DPGZ:92,BDP:06,L:09,DPZ:96,DP:04,ESS:08,ES:09,CMG:95,M:89,ESSTvG:13,S:14,S:13}. See~\cite{S:11-1,COTX:24} for survey articles on the topic of ergodicity and Kolmogorov operators for SPDEs. For singular drift monotone SPDEs with degenerate Gaussian noise, ergodicity has been obtained in~\cite{GT:14,ESR:12,GT:16,SST:23,LT:11,W:15,ESvRS:12,BDP:10,L:11-2}. For SPDEs with jump noise, ergodicity has for instance been discussed in \cite{FJKR:23,PZZ:24}. 
See~\cite{SZ:24} for a recent result on ergodicity of the stochastic Allen-Cahn equation with logarithmic potential.

For the stochastic Navier-Stokes equations, ergodicity has been first obtained in~\cite{HM:06,M:02,KS:01,EMS:01,FM:95,BKL:02,GM:05}, where the results have been further refined in \cite{DX:11,BHZ:13,BDP:08,S:11-2,KNS:20,HM:11,RX:11,M:16,MSS:19,ADX:12,KS:12,GHMR:17}. Up to now, questions of ergodicity, together with questions from turbulence and advection for stochastic Navier-Stokes equations with non-degenerate noise are an active topic of research, see~\cite{ARMA,ARMA2,HZZ:23,FM:12,CMP,BKS:20,FZ:24,HZ:24,BKP:22,BPZ:23,HZZ:24,PZZ:24,CR:24,NTT:21}. See \cite{GHMR:17,FZ:25,COTX:24} for a survey of the results.

We shall consider additive noise SPDEs of the type
\[\ud X^x_t=A(X^x_t)\,\ud t+B\,\ud W_t,\quad X^x_0=x,\]
where the initial datum $x$ lies in a separable Hilbert space $H$, and $\{W_t\}_{t\ge 0}$ is a cylindrical Wiener process. The detailed assumptions can be found in Sections~\ref{sec:model} and \ref{sec:hyp}.
The setup for SPDEs with \emph{locally monotone drift} has been introduced by Liu and R\"ockner in~\cite{LR:10}. A nonlinear drift operator $A:V\to V^\ast$, where $V\subset H$ is a reflexive Banach space with dual $V^\ast$, is called \emph{locally monotone} if for every $u,v\in V$,
\[\langle A(u)-A(v),u-v\rangle\le (K+ \rho(u))\|u-v\|^2_H,\]
where $K\in\R$, and $\rho:V\to [0,\infty)$ is locally bounded and Borel measurable on $V$. If $\rho$ depends on both $u$ and $v$ in a nontrivial way, $A$ is called \emph{fully locally monotone}.
As a unifying approach, it has sparked a lot of interest and has lead to many subsequent works, see~\cite{LR:15,AV:24,BLZ:14,GHV:22,NS:20,RSZ:24,KM:24,AV:22-1,AV:22-2,GLS:20}, among many other contributions. Examples for well-posedness covered by the approach of locally monotone drift SPDEs are the stochastic incompressible 2D Navier-Stokes equations, stochastic incompressible power law fluids, the stochastic incompressible tamed 3D Navier-Stokes equations, the stochastic 3D Leray-$\alpha$ model, as well as, quasilinear equations as the stochastic $p$-Laplace equation with non-monotone perturbation, and, with certain restrictions on the growth behavior and spatial dimension, semilinear stochastic equations as the 1D stochastic Burgers equation, and the 1D stochastic Allen-Cahn equation with double-well potential, and a large class of reaction-diffusion equations.

In this work, our main aim is to provide an abstract framework for ergodicity of locally monotone SPDEs with degenerate additive Wiener noise. Existence and uniqueness of solutions in our setting have been obtained by Liu and R\"ockner in \cite{LR:10} and have been extended by R\"ockner, Shang and Zhang in \cite{RSZ:24}. As our main example, we prove the ergodicity and mixing of the stochastic incompressible 2D Navier-Stokes equations with no-slip boundary conditions on a Poincar\'e domain or a domain with periodic boundary conditions on a square for large enough viscosity, precisely quantified, without any spatial non-degeneracy condition on the noise,
thus including the deterministic incompressible 2D Navier-Stokes equations. The examples of stochastic semilinear equations include the stochastic Burgers equation in 1D, however, due to our conditions, only small quadratic growth perturbations of the second order differential operator are permitted, thus ruling out the stochastic Allen-Cahn equation in 1D with double-well potential. This is not a surprise, however, as generally speaking, the locally monotone framework of Liu and R\"ockner \cite{LR:10,LR:15} is better suited for SPDEs where the highest order term is pseudo-monotone or strongly dissipative, as the construction of solutions is based on Faedo-Galerkin approximations and weak convergence methods, whereas semilinear drift SPDEs are often constructed via fixed point arguments, as for instance in the approach of Agresti and Veraar \cite{AV:24,AV:22-1,AV:22-2}, which uses critical exponents and a local Lipschitz condition.
Ergodicity for locally monotone drift SPDEs has been obtained in~\cite{Z:19} for non-degenerate noise.

\subsection*{The methods}

Our main result Theorem~\ref{thm:main} below is proved in 
Section~\ref{sec:proof}. The proof consists of the following steps.
First, we derive certain a priori estimates for the solutions. These lead to the $e$-property of the semigroup, that is, the uniform-in-time equicontinuity of the Feller semigroup $P_t F(x):=\E[F(X_t^x)]$, applied to a bounded Lipschitz function $F:H\to\R$. It can be viewed as a coupling condition at infinity. In many cases, the $e$-property is implied by the asymptotic strong Feller property by duality arguments \cite{SW:12}. Together with weak topological irreducibility, the $e$-property implies uniqueness of the invariant measure \cite{KSS:12}. For monotone drift SPDEs with additive noise, the $e$-property follows trivially from Gronwall's lemma. Similarly, by the standard Gronwall lemma, one obtains an estimate of the form
\begin{align*}&|P_t F(x)-P_t F(y)|^2\le\operatorname{Lip}(F)^2\; \E\left[\|X_t^x-X_t^y\|_H^2\right]\\
\le&\operatorname{Lip}(F)^2\;\|x-y\|^2_H\; \E\left[\exp\left(Kt+\int_0^t\rho(X_s^x)\,\ud s\right)\right].\end{align*}
Under suitable conditions on $K< 0$ and $\rho$, the time integral can be bounded pathwise by local monotonicity and It\^o's lemma by a function of the norm of the initial datum $x$ and a (local) martingale, and some lower order terms of bounded variation. The exponential moment of the martingale can be controlled by a multiple of the exponential of the quadratic variation. A
similar estimate has been proposed by Butkovsky, Kulik and Scheutzow in \cite{BKS:20,KS:18} for a general setup for couplings, however for non-degenerate noise.

The next step is to compare the paths of the solutions to the solution to the deterministic PDE with noise set to zero, which has an exponential or polynomial decay behavior due to our coercivity assumptions.
The conditional coupling will be proved  with the help of the stochastic Gronwall Lemma~\cite{G:24} by Geiss and the small ball property of the noise.
This coupling leads to an asymptotic irreducibility property which yields the lower bound condition of Komorowski, Peszat and Szarek~\cite{KPS:10}. This procedure to prove ergodicity has been first followed by Es-Sarhir and von Renesse in~\cite{ESR:12}, and has been applied to singular multi-valued monotone SPDEs by Gess and the second author in~\cite{GT:14}.

Given a prescribed error $\e>0$, we define the $\e$-mixing time (with respect to the Wasserstein-2-distance $\mathcal{W}_2$ on the space of probability measures with second moments) as 
\begin{equation}
\tau^x_{\textsf{mix}}(\e):=\inf\{t\geq 0: \mathcal{W}_2(P_t(x,\cdot),\mu_\ast)\leq \e\},
\end{equation}
where $\mu_\ast$ is the unique invariant measure and $P_t(x,A):=P_t\mathbbm{1}_A(x)$.
Under certain quantitative conditions, we obtain an explicit upper bound for the $\e$-mixing time by using our improved moment bound and the exponential convergence of the semigroup, see for instance~\cite{DIA:96,LPW:09} for the definition of mixing times in the context of Markov chains.

\subsection*{Main contributions}
The main contributions of this work, in relation to the existing literature, are the following.
\begin{enumerate}[(a)]
\item \textbf{Ergodicity under arbitrary degeneracy of the noise.}
We establish exponential ergodicity and uniqueness of the invariant measure (Theorem~\ref{thm:main}) without any non-degeneracy assumption on the noise coefficient $B$; in particular, $B=0$ is admissible. As a consequence, the deterministic limit equation is contained as a corollary, and we recover the stability and extinction results for the deterministic 2D Navier--Stokes equations of Temam~\cite[Chapter~10]{T:01} (see Remark~\ref{rem:detNSE} in Subsection~\ref{subsec:ex3}). This is in contrast to the asymptotic-strong-Feller approach of Hairer and Mattingly~\cite{HM:06,HM:11}, which requires at least four independent forced Fourier modes, and to the generalized-coupling approach of Butkovsky, Kulik and Scheutzow~\cite{BKS:20,KS:18}, which also requires non-degenerate noise; we note that our result, in turn, requires the viscosity to be large relative to the noise intensity, see Hypothesis~\ref{hyp:quant}.
However, we do not need to assume anything on the range of our noise operator, so the result for the partial differential equation (PDE) without noise is a natural special case.

\item \textbf{No compact-embedding assumption on the energy space.}
Our framework dispenses with compactness of the embedding $V\subset H$, which is standard in the Krylov--Bogoliubov method~\cite{DPZ:96} and in much of the locally-monotone-drift literature~\cite{LR:10,LR:15,BLZ:14,RSZ:24}. This permits unbounded spatial domains, namely \emph{Poincar\'e domains} (in particular, domains of finite width), where the Sobolev embedding is not necessarily compact.

\item \textbf{Main method: the $e$-property.}
Uniqueness of the invariant measure is obtained through the lower-bound technique of Lasota and Szarek~\cite{LS:06} and its refinement by Komorowski, Peszat and Szarek~\cite{KPS:10}, the key structural ingredient being the $e$-property of the semigroup (Proposition~\ref{prop:e}). This technique requires neither tightness, as does the Krylov--Bogoliubov method~\cite{DPZ:96}, nor the sequential weak Feller property, as in Maslowski and Seidler~\cite{MS:99}. The technique is also related to the approach of Es-Sarhir and von Renesse~\cite{ESR:12}, applied to singular multi-valued monotone SPDEs by Gess and the second author~\cite{GT:14}. The trade-off is that we obtain weak$^\ast$-mean ergodicity rather than total-variation convergence; under quantitative smallness assumptions (Hypothesis~\ref{hyp:mixing}), we additionally obtain exponential mixing.

\item \textbf{Quantitative $\varepsilon$-mixing-time bounds in the Wasserstein-2 distance.}
Under quantitative smallness of $\|B\|_{L_2(U,H)}$ relative to the viscosity (Hypothesis~\ref{hyp:mixing}), we provide explicit upper bounds for the $\varepsilon$-mixing time $\tau^{x}_{\textsf{mix}}(\varepsilon):=\inf\{t\ge 0\,:\,\mathcal{W}_2(P_t(x,\cdot),\mu_\ast)\le\varepsilon\}$ (Theorems~\ref{thm:main2} and~\ref{thm:mixing}). The dependence on the initial datum, the noise intensity and $\varepsilon$ is fully explicit. To the best of our knowledge, such explicit mixing-time bounds in $\mathcal{W}_2$ are new at this level of generality, in particular for the stochastic 2D Navier--Stokes equations with degenerate additive noise. We do not claim that our results are optimal.

\item \textbf{Unified treatment of degenerate-noise examples on possibly unbounded domains.}
The abstract framework applies, in a unified way, to: the stochastic incompressible 2D Navier--Stokes equations on Poincar\'e domains (Theorem~\ref{thm:2DNSE}); shear-thickening stochastic power-law fluids in dimension $d\ge 2$ with $p= 1+\tfrac{d}{2}$, covering both $p=d=2$ and $(p,d)=(\tfrac{5}{2},3)$ (Theorem~\ref{thm:plaw}); the stochastic heat equation (Theorem~\ref{thm:heat}); and semilinear stochastic equations including small quadratic perturbations of the 1D stochastic Burgers equation (Theorem~\ref{thm:semi}). The 2D Navier--Stokes case is handled by a decoupling of the variational parameters into $(\alpha,\beta)=(2,0)$ for the ergodicity estimate---where the antisymmetry $\langle F(u),u\rangle=0$ eliminates the quadratic contribution from the energy inequality---and $(\hat\alpha,\hat\beta)=(2,2)$ for well-posedness in the framework of Liu--R\"ockner~\cite{LR:10,LR:15} and Brze\'zniak--Liu--Zhu~\cite{BLZ:14}. Other possible examples could be 2D third-grade fluids on unbounded domains \cite{KC:25-1}.
\end{enumerate}

\subsection*{Organization of the paper}
In Section~\ref{sec:model}, we present our model and the main results of this paper. In Section~\ref{sec:hyp}, we shall pose and discuss the main hypotheses for our results, namely 
Hypotheses~\ref{hyp:hemicont},~\ref{hyp:coerc}, \ref{hyp:coerc-g}, \ref{hyp:monotonicity}, \ref{hyp:monotonicity-g}, \ref{hyp:growth}, \ref{hyp:angle}, \ref{hyp:regularity}, \ref{hyp:reg2}, \ref{hyp:quant}, and \ref{hyp:mixing}. Subsequently, we discuss the necessary background from Markovian semigroups and invariant measures, together with the most important auxiliary results. In Section~\ref{sec:examples}, we prove Theorem~\ref{thm:heat}, Theorem~\ref{thm:semi}, Theorem~\ref{thm:2DNSE}, and Theorem~\ref{thm:plaw}. 
In Section~\ref{sec:proof}, we prove the main ergodicity Theorem~\ref{thm:main}.  In Section~\ref{sec:mixing}, we prove Theorem~\ref{thm:main2} and Theorem~\ref{thm:mixing} on mixing times.

\subsection*{Notation} For a metric space $\mathcal{X}$, denote the continuous and bounded real-valued functions on $\mathcal{X}$ by $C_b(\mathcal{X})$, equipped with the supremum norm $\|f\|_\infty:=\sup_{x\in \mathcal{X}}|f(x)|$. Denote the Lipschitz continuous functions from $\mathcal{X}$ to $\R$ by $\Lip(\mathcal{X})$, and denote the bounded and Lipschitz continuous functions from $\mathcal{X}$ to $\R$ by $\Lip_b(\mathcal{X})$. For $f\in\Lip(\mathcal{X})$, denote the Lipschitz constant of $f$ by $\Lip(f)$.  $\Lip_b(\mathcal{X})$ is equipped with the norm $\|f\|_{\Lip_b}:=\Lip(f)+\|f\|_\infty$. Denote by $|\cdot|$, $\cdot$, respectively, the Euclidean norm on $\R^d$, and the Euclidean scalar product on $\R^d\times\R^d$, respectively. The transpose of a real vector or matrix is denoted by the upper index $^{\textup{t}}$. The adjoint of a linear operator on a Hilbert space is denoted by the upper index $^{\ast}$. For a domain $\Ocal\subset\R^d$, $d\in\N$, we denote by $W_0^{1,p}(\Ocal;\R^k)$, $k\in\N$, the closure of compactly supported smooth functions $C_0^\infty(\Ocal;\R^k)$ in $L^p(\Ocal;\R^k)$, with respect to the Sobolev norm $\|v\|_{1,p}:=\left(\int_\Ocal |\nabla v|^p\,\ud x\right)^{1/p}$. For Banach spaces $V_1,V_2$, denote space of linear operators from $V_1$ to $V_2$ by $L(V_1,V_2)$ with operator norm $\|\cdot\|_{L(V_1,V_2)}$. For separable Hilbert spaces $H_1,H_2$, denote the space of Hilbert-Schmidt operators from $H_1$ to $H_2$ by $L_2(H_1,H_2)$ with Hilbert-Schmidt norm $\|\cdot\|_{L_2(H_1,H_2)}$. As usual, $a\wedge b$ denotes the minimum of two real numbers $a$ and $b$, and $a\vee b$ denotes the maximum of two real numbers $a$ and $b$.

\section{The model and main results}\label{sec:model}

We are interested in SPDEs with locally monotone drift and additive Wiener noise.
Let $(\Omega,\Fcal,(\Fcal_t)_{t\ge 0},\P)$ be a filtered probability space satisfying the standard conditions. Let $H$ be a separable Hilbert space. Let $V$ be a reflexive Banach space embedded linearly, densely, and continuously to $H$. We would like to point out that we generally do not assume compactness of the embedding which enables us to consider certain unbounded spatial domains, namely, Poincar\'e domains\footnote{By definition, a \emph{Poincar\'e domain} is a domain that satisfies the Poincar\'e inequality. Domains of \emph{finite width} are examples of Poincar\'e domains. A domain $\Ocal\subset\R^d$ of \emph{finite width} fits by definition between two parallel $(d-1)$-dimensional hyperplanes, see \cite{AF:03}. Sometimes it is referred to as a domain which is bounded in one direction.}.

We consider the unique strong solution $(X_t)_{t\geq 0}$ of  
\begin{equation}\label{eq:model}
\ud X_t=A(X_t)\,\ud t+B\,\ud W_t,\quad X_0=x\in H,
\end{equation}
where $(W_t)_{t\geq 0}$ is a $U$-valued cylindrical Wiener process on $(\Omega,\Fcal,(\Fcal_t)_{t\ge 0},\P)$ for some separable Hilbert space $U$, see~\cite{DPZ:14}. The
locally monotone nonlinear drift operator $A:V\to V^\ast$ and the bounded linear operator $B:U\to H$ satisfy the hypotheses in Section~\ref{sec:hyp}, in particular, it has finite Hilbert-Schmidt norm $\|B\|_{L_2(U,H)}<\infty$.
Note that $B$ can be equal to zero.

\subsection{Main results}

Our main result is the following abstract ergodicity result for SPDEs of the type~\eqref{eq:model} with locally monotone drift with degenerate and spatially regular Wiener noise. See Subsection~\ref{subsec:ergo} for the precise definitions of the terminology.

\begin{theorem}\label{thm:main}
Assume that Hypotheses~\ref{hyp:hemicont},~\ref{hyp:coerc},~\ref{hyp:coerc-g},~\ref{hyp:monotonicity},~\ref{hyp:monotonicity-g},~\ref{hyp:growth},~\ref{hyp:angle},~\ref{hyp:regularity},~\ref{hyp:reg2}, and \ref{hyp:quant} given in Section~\ref{sec:hyp} hold true. Then, the stochastically continuous Markovian Feller semigroup $(P_t)_{t\ge 0}$ associated to~\eqref{eq:model} satisfies the $e$-property in $H$ and is weak$^\ast$-mean ergodic. Moreover, it admits a unique invariant probability measure $\mu_\ast$ on $(H,\Bcal(H))$ that admits finite $(\alpha+\beta)$-moments in $H$ and finite $\alpha$-moments in $V$, where $\alpha\ge 2$ is as in Hypothesis~\ref{hyp:coerc} and $\beta\ge 0$ is as in Hypothesis~\ref{hyp:monotonicity}.
\end{theorem}
\begin{proof}
See Section~\ref{sec:proof}.
\end{proof}

Under the additional assumption that $V\subset H$ is a compact embedding (which we do not assume), Hypothesis~\ref{hyp:coerc} already guarantees that the existence of at least one invariant measure, which can be proved by the classical method of Krylov and Bogoliubov, see~\cite[Theorem 3.1.1]{DPZ:96}.

See Section~\ref{sec:mixing} for the terminology of $\eps$-mixing times.

\begin{theorem}\label{thm:main2}
Assume that Hypotheses~\ref{hyp:hemicont},~\ref{hyp:coerc},~\ref{hyp:coerc-g},~\ref{hyp:monotonicity},~\ref{hyp:monotonicity-g},~\ref{hyp:growth},~\ref{hyp:angle},~\ref{hyp:regularity},~\ref{hyp:reg2},~\ref{hyp:quant}, and \ref{hyp:mixing} given in Section~\ref{sec:hyp} hold true. Then, the stochastically continuous Markovian Feller semigroup $(P_t)_{t\ge 0}$ associated to~\eqref{eq:model} admits an exponentially ergodic unique invariant probability measure $\mu_\ast$ on $(H,\Bcal(H))$ such that the $\eps$-mixing time of $\mu_\ast$ in Wasserstein-2-distance is bounded above. See Theorem~\ref{thm:mixing} for the explicit bound.
\end{theorem}
\begin{proof}
See Section~\ref{sec:proof2}.
\end{proof}

In Section~\ref{sec:examples}, we discuss the following examples.
First, we consider the stochastic heat equation on a Poincar\'e domain $\Ocal\subset\R^d$, $d\in\N$, $\nu>0$, with Dirichlet boundary conditions on a Lipschitz boundary,
\begin{equation}\label{eq:heat}\ud X_t=\nu\Delta X_t\,\ud t+B\,\ud W_t,\quad X_0\in L^2(\Ocal).\end{equation}
We report the following result, compare also with~\cite{BDP:06,DPZ:96}.
\begin{theorem}\label{thm:heat}
Assume that $B\in L_2(U,W_0^{1,2}(\Ocal))$. Denote by $c_0$ the inverse Poincar\'e constant of $\Ocal$, and assume that for some $\lambda\in (0,1)$,
\[\|B\|^2_{L_2(U,H)}\le \lambda\nu c_0^2.\]
Then the Markovian Feller semigroup $(P_t)_{t\ge 0}$ of the stochastic heat equation~\eqref{eq:heat} with degenerate additive Wiener noise is exponentially ergodic and exponentially mixing and possesses a unique invariant measure $\mu_\ast$ with finite second moments in the stronger space $V$ in all spatial dimensions. We have the following upper bound for the $\e$-mixing time 
\[
\tau^x_{\textup{\textsf{mix}}}(\e)\leq 
\frac{1}{\nu c_0^2}\left[\log\left(\|x\|_H+\frac{\|B\|_{L_2(U,H)}}{\sqrt{2(1-\lambda)\nu}c_0}\right)+\log\left(\frac{1}{\e}\right)\right].
\]
\end{theorem}
\begin{proof}
See Subsection \ref{subsec:ex1}.
\end{proof}

Furthermore, we consider the semilinear stochastic equation on a Poincar\'e domain $\Ocal\subset\R^d$, $d=1,2$, $\nu>0$, with Lipschitz boundary,
\begin{equation}\label{eq:semi}\ud X_t=\left(\nu\Delta X_t+\mathbf{f}(X_t)\cdot\nabla X_t+g(X_t)\right)\,\ud t+B\,\ud W_t,\quad X_0\in L^2(\Ocal).
\end{equation}
For $\mathbf{f}:=(f_1,\ldots,f_d)^{\textup{t}}$, where $v^{\textup{t}}$ denotes the transpose of $v$, we assume that $f_1(x)=x$ in $d=1$ or that $f_i\in\Lip_b(\R)$, $i=1,\ldots,d$ for $d=1,2$.
Let $c_0>0$ be the inverse Poincar\'e constant of $\Ocal$.
We assume that $g:\R\to\R$ is continuous with $g(0)=0$, and that there exist non-negative constants $C,c,s,K,k$, such that
\begin{equation}\label{eq:g1}|g(x)|\le C\left(1+|x|^2\right),\quad x\in\R,\end{equation}
and that
\begin{equation}\label{eq:g2}\left(g(x)-g(y)\right)(x-y)\le c\left(1+|y|^s\right)(x-y)^2,\quad x,y\in\R,\end{equation}
such that $s\le 2$.
For $d=1$, assume that $c<\nu c_0^2$. For $d=2$, assume that
\begin{equation}\label{eq:g5}\frac{c}{2\nu c_0^2}+\frac{\|\mathbf{f}\|_{L^\infty}}{\nu c_0}<\frac{1}{2}.\end{equation}
Furthermore, we assume that
\begin{equation}\label{eq:g3}g(x)x\le K+ k|x|^2,\quad x\in\R.\end{equation}
Assume also that
\begin{equation}\label{eq:g4}\frac{\|\mathbf{f}\|_{L^\infty}}{\nu c_0}+\frac{k}{\nu c_0^2}<2.\end{equation}
We impose Dirichlet boundary conditions. We get the following result.
\begin{theorem}\label{thm:semi}
Assume that $\mathbf{f}$ and $g$ satisfy~\eqref{eq:g1},~\eqref{eq:g2},~\eqref{eq:g5},~\eqref{eq:g3}~and~\eqref{eq:g4} and assume that $B\in L_2(U,W_0^{1,2}(\Ocal))$. Assume that Hypotheses~\ref{hyp:quant}~and~\ref{hyp:mixing} hold. Then the Markovian Feller semigroup $(P_t)_{t\ge 0}$ of the stochastic semilinear equations of the form~\eqref{eq:semi} for $d=1,2$, with degenerate additive Wiener noise is exponentially mixing and exponentially ergodic and possesses a unique invariant measure $\mu_\ast$ with finite second $V$-moments. In particular, this holds for the stochastic Burgers equation in 1D for $f_1(x)=x$.
\end{theorem}
\begin{proof}
See Subsection \ref{subsec:ex2}.
\end{proof}

For the stochastic incompressible 2D Navier-Stokes equations on a Poincar\'e domain $\Ocal\subset\R^2$ with no-slip (that is, Dirichlet) boundary conditions on a Lipschitz boundary $\partial\Ocal$, with viscosity $\nu>0$,
\begin{equation}\label{eq:2DNSE}
\ud X_t=\left(\nu \mathbf{P}\Delta X_t-\mathbf{P}\left[(X_t\cdot\nabla) X_t\right]\right)\,\ud t+B\,\ud W_t,\quad X_0\in L^2_{\text{sol}}(\Ocal;\R^2),
\end{equation}
where $\mathbf{P}$ denotes the Helmholtz-Leray projection on the solenoidal fields, and $L^2_{\operatorname{sol}}(\Ocal):=\mathbf{P}(L^2(\Ocal))$, $W^{1,2}_{0,\,\operatorname{sol}}(\Ocal):=\mathbf{P}(W^{1,2}_0(\Ocal))$. Let $c_0>0$ be the inverse Poincar\'e constant of $\Ocal$.
We get the following result, which is, to our knowledge, novel for degenerate noise.
\begin{theorem}\label{thm:2DNSE}
Assume that $B\in L_2(U,W^{1,2}_{0,\,\operatorname{sol}}(\Ocal))$. Assume that for some $\lambda\in (0,1)$,
\[\|B\|^2_{L_2(U,H)}\le\frac{1}{4}\lambda\nu^3 c_0^2,\]
and
\[\|B\|^2_{L_2(U,H)}\le(1-\lambda)\nu c_0^2.\]
Then the Markovian Feller semigroup $(P_t)_{t\ge 0}$ of the stochastic incompressible 2D Navier-Stokes equations~\eqref{eq:2DNSE} with degenerate additive Wiener noise possesses a unique invariant measure $\mu_\ast$ with second moments in $V$.

Furthermore, if there exists $\gamma\in (0,\nu c_0^2]$ and $\lambda\in (0,1)$ such that
\[\|B\|^2_{L_2(U,H)}\le\frac{1}{4}\lambda(\nu^3 c_0^2-\nu^2\gamma),\]
and
\[\|B\|^2_{L_2(U,H)}\le(1-\lambda)\nu c_0^2,\]
then $(P_t)_{t\ge 0}$ is exponentially mixing and exponentially ergodic, and we have the following upper bound for the $\e$-mixing time
\[\tau^x_{\textup{\textsf{mix}}}(\e)\leq 
\frac{2}{\gamma}\left[\frac{1}{\lambda \nu^2 }\|x\|^{2}_H+\log\left(\|x\|_H+\frac{\|B\|_{L_2(U,H)}}{\sqrt{2\lambda \nu}c_0}\right)+\log\left(\frac{1}{\e}\right)\right].\]
\end{theorem}
\begin{proof}
See Subsection \ref{subsec:ex3}.
\end{proof}

As a generalization, one can consider the velocity field of a viscous and incompressible non-Newtonian fluid perturbed by Wiener noise with Dirichlet boundary conditions on a Poincar\'e domain $\Ocal\subset\R^d$ with Lipschitz boundary $\partial\Ocal$, $d\ge 2$. Let $p=1+\frac{d}{2}$, $\nu>0$, and $u:\Ocal\to\R^d$. Define
\[e(u):\Ocal\to\R^d\otimes\R^d,\quad e_{i,j}(u):=\frac{\partial_i u_j+\partial_j u_i}{2},\quad 1\le i,j\le d,\]
and
\[
\tau(u):\Ocal\to\R^d\otimes\R^d,\quad \tau(u):=2\nu(1+|e(u)|)^{p-2}e(u),
\]
where $\otimes$ denotes the usual tensor product.
Consider the stochastic power-law fluid equations
\begin{equation}\label{eq:plaw}
\ud X_t=\left(\mathbf{P}(\operatorname{div}(\tau(X_t)))-\mathbf{P}\left[(X_t\cdot\nabla) X_t\right]\right)\,\ud t+B\,\ud W_t,\quad X_0\in L^2_{\text{sol}}(\Ocal;\R^d).
\end{equation}
We obtain the following result.
\begin{theorem}\label{thm:plaw}
Assume $p=1+\frac{d}{2}$ for $d\ge 2$, and that Hypotheses~\ref{hyp:regularity},~\ref{hyp:reg2},~\ref{hyp:quant} hold true. Then the Markovian Feller semigroup $(P_t)_{t\ge 0}$ of the stochastic shear thickening power-law fluid equations~\eqref{eq:plaw} with degenerate additive Wiener noise possesses a unique invariant measure $\mu_\ast$ with finite $p$-moments in $W^{1,p}_{\operatorname{sol},0}(\Ocal;\R^d)$. If, additionally, Hypothesis~\ref{hyp:mixing} holds, then $(P_t)_{t\ge 0}$ is exponentially ergodic and exponentially mixing, with a quantitative upper bound for the $\e$-mixing time given in Theorem~\ref{thm:mixing}.
\end{theorem}
\begin{proof}
See Subsection \ref{subsec:ex4}.
\end{proof}

It is easy to see that the situation $p=d=2$, that is, the 2D Navier-Stokes case, as well as e.g. $p=\frac{5}{2}$ and $d=3$ are covered in the assumptions of Theorem~\ref{thm:plaw}. However, the 3D Navier-Stokes case (that is, $p=2$ and $d=3$) is not covered. Compare also with the results of  
\cite{BLZ:14} and \cite{S:13}.

We note that our results remain true for periodic boundary conditions, that is, when one replaces $\Ocal$ by the flat torus $\mathbb{T}^d$ in all of the examples above, see Remark \ref{rem:torus}.

\section{Preliminaries and hypotheses}\label{sec:hyp}

The following hypotheses are modifications and partial extensions, needed for our ergodicity result, of the hypotheses for the variational well-posedness result for SPDEs with locally monotone drift in finite time from~\cite{LR:10,LR:15} for Gaussian noise, which has been extended to L\'evy noise in~\cite{BLZ:14,NTT:21}. See also~\cite{GHV:22,NS:20,AV:24,KM:24,RSZ:24,AV:22-1,AV:22-2} for further extensions, which cover examples that are out of the scope of this paper; for instance the 3D stochastic tamed Navier-Stokes equations, the stochastic $p$-Laplace equation, the stochastic Allen-Cahn equation, and the stochastic Cahn-Hilliard equation, where the drift of the latter SPDE satisfies bounds which have been called \emph{fully locally monotone} in \cite{RSZ:24}.

Let $H$ be a separable Hilbert space. Let $V$ be a reflexive Banach space embedded linearly, densely, and continuously to $H$. Note that we identify $H$ with its Hilbert space dual $H^\ast$ by the Riesz isometry, so that we obtain a Gelfand triple
\[V\subset H\equiv H^\ast\subset V^\ast,\]
where $V^\ast$ denotes the topological dual of $V$.
We shall use the notation $\langle u,v\rangle$ both for $u\in V$ and $v\in V^\ast$, where it denotes the evaluation of a dual element, and for $u,v\in H$, where it denotes the Hilbert space inner product of $u$ and $v$ in $H$, where, in particular, for any $u\in V$ is considered as an element of $H$ by the continuous embedding, and in the latter case both meanings of the notation coincide.

\begin{fact}[Embedding]\label{rem:embedding}
As the embedding $V\subset H$ is linear and bounded, there exists a constant $c_0>0$ such that
\[
\|x\|_V\geq c_0 \|x\|_H \quad \textrm{ for any } \quad x\in V.
\]
\end{fact}

\subsection{Hypotheses on the drift operator}

Consider the following set of hypotheses.

\begin{hypothesis}[Hemicontinuity]\label{hyp:hemicont}
For every $x,y,z\in V$ the map 
\[\mathbb{R}\ni\lambda\mapsto \langle A(x+\lambda y),z\rangle
\quad \textrm{ is continuous.}
\]
\end{hypothesis}

\begin{hypothesis}[Coercivity]\label{hyp:coerc}
There exist constants $\delta_1>0$, and $\alpha \ge 2$, such that
\[
2\langle  A(x),x\rangle \leq -\delta_1 \|x\|^\alpha_V \quad \textrm{ for all } \quad x\in V. 
\]
\end{hypothesis}

\begin{hypothesis}[Coercivity, general form]\label{hyp:coerc-g}
There exist constants $K_1\in\R$, $\hat{\delta}_1>0$ and $\hat{\alpha} \ge 2$, such that
\[
2\langle  A(x),x\rangle \leq K_1(1+\|x\|^2_H) -\hat{\delta}_1 \|x\|^{\hat{\alpha}}_V \quad \textrm{ for all } \quad x\in V. 
\]
\end{hypothesis}

\begin{hypothesis}[Full local monotonicity]\label{hyp:monotonicity}
There exist constants $\delta_2> 0$, $C_2\geq 0$, and $\beta\geq 0$ such that
\[
2\langle A(x)-A(y),x-y \rangle \leq [-\delta_2+\eta(x)+\rho(y)] \|x-y\|^2_H\quad \textrm{ for all }\quad x, y\in V,
\]
where $\alpha\ge 2$ is given in Hypothesis~\ref{hyp:coerc}, and $\rho,\eta:V\to [0,\infty)$ are measurable functions, satisfying
\[
0\leq \eta(x)+\rho(x)\leq C_2\|x\|^\alpha_V \|x\|^\beta_H\quad \textrm{ for any }\quad x\in V.
\]
If both $\eta$ and $\rho$ are non-trivial, we also assume that the embedding $V\subset H$ is compact.
\end{hypothesis}

If both $\eta$ and $\rho$ in Hypothesis~\ref{hyp:monotonicity} are non-trivial, we say that we are in the \emph{fully local monotone} case, see \cite{RSZ:24}. If Hypothesis~\ref{hyp:monotonicity} holds with either $\eta\equiv 0$ or $\rho\equiv 0$, we say that we are in the 
\emph{locally monotone} case, see \cite{LR:15}. If Hypothesis~\ref{hyp:monotonicity} holds with $\eta=\rho\equiv 0$, we say that we are in the \emph{strongly monotone} case, see \cite{KR:79}.

\begin{hypothesis}[Full local monotonicity, general form]\label{hyp:monotonicity-g}
There exist constants $K_2\in\R$, $C_2\geq 0$, and $\hat{\beta}\geq 0$ such that
\[
2\langle A(x)-A(y),x-y \rangle \leq [K_2+\eta(x)+\rho(y)] \|x-y\|^2_H\quad \textrm{ for all }\quad x, y\in V,
\]
where $\hat{\alpha}\ge 2$ is given in Hypothesis~\ref{hyp:coerc-g}, and $\rho,\eta:V\to [0,\infty)$ are measurable functions, satisfying
\[
0\leq \eta(x)+\rho(x)\leq C_2\|x\|^{\hat{\alpha}}_V \|x\|^{\hat{\beta}}_H\quad \textrm{ for any }\quad x\in V.
\]
If $\rho\equiv 0$ (or $\eta\equiv 0$) we assume that $\eta$ (or $\rho$, respectively) is hemicontinuous, i.e. the map $\R\ni\lambda\mapsto\eta(x+\lambda y)$ is continuous for any $x,y\in V$.
If both $\eta$ and $\rho$ are non-trivial, we also assume that the embedding $V\subset H$ is compact.
\end{hypothesis}

\begin{hypothesis}[Growth]\label{hyp:growth}
There exists a constant  $K_3>0$ satisfying
\begin{equation}\|A(x)\|_{V^\ast}^{\frac{\hat{\alpha}}{\hat{\alpha}-1}}\le K_3(1+\|x\|_V^{\hat{\alpha}})(1+\|x\|_H^{\hat{\beta}})\quad
\textrm{ for all } \quad x\in V,
\end{equation}
where $\hat{\alpha}$ and $\hat{\beta}$ are given in Hypotheses~\ref{hyp:coerc-g}~and~\ref{hyp:monotonicity-g}, respectively.
\end{hypothesis}

\begin{hypothesis}[Cone condition]\label{hyp:angle}
There exist constants $\delta_4>0$ and $C_4\in\R$ such that
\begin{equation}2\langle A(x),x\rangle\le C_4-\delta_4\|A(x)\|_{V^\ast}
\quad \textrm{ for any }\quad x\in V.
\end{equation}
\end{hypothesis}

Note that Hypothesis~\ref{hyp:angle} is implied by
\[
2\langle A(x),x\rangle\le C_4-\delta_4\|A(x)\|_{V^\ast}^q
\quad \textrm{ for any }\quad x\in V
\]
for some $q\ge 1$.

\begin{remark}\begin{enumerate}
\item Existence and uniqueness of solutions will use another set of hypotheses than the ergodicity result, which explains why we distinguish between $\alpha$ and $\hat{\alpha}$, as well as between $\beta$ and $\hat{\beta}$. In many cases, one may assume $\alpha=\hat{\alpha}$ and $\beta=\hat{\beta}$, but not in the prominent case of the stochastic 2D Navier-Stokes equations, see Subsection \ref{subsec:ex3} below.
\item
Note that especially 
Hypothesis~\ref{hyp:monotonicity} and Hypothesis~\ref{hyp:angle} differ from the standard assumptions for locally monotone drift SPDEs~\cite{LR:15}. To obtain global coupling at infinity (the $e$-property, see Definition~\ref{def:e} below), we need to obtain a priori estimates independent of the terminal time $T>0$. To get this, we need to assume $\delta_2>0$, and we cannot permit constant positive perturbations of the drift. Essentially, this means that our drift has a strongly dissipative part.
Hypothesis~\ref{hyp:angle} is a cone condition that is of technical nature, but is easy to verify for many examples where the growth of the lower order nonlinearities can be controlled by the potential of the leading order term, see the examples in Section~\ref{sec:examples} below. Hypothesis~\ref{hyp:angle} can for instance be verified with the help of the conditions of Lemma~\ref{lem:subgradient} and Lemma~\ref{lem:hypE} below. Note that we need Hypothesis~\ref{hyp:angle} only in the proof of Lemma~\ref{lem:coupling} below.
\end{enumerate}
\end{remark}

\subsection{Hypotheses on the noise}

\begin{hypothesis}[$H$-regularity of the noise]\label{hyp:regularity}
The noise coefficient $B:U\to H$ is linear, bounded and satisfies
\begin{equation}
B\in L_2(U,H).
\end{equation}
\end{hypothesis}

\begin{hypothesis}[$V$-Regularity of the noise]\label{hyp:reg2}
For any $T>0$, the driving process $\{B\,W_t\}_{t\in [0,T]}$ satisfies a small ball property in $V$, that is, for any $\delta>0$,
\begin{equation}
\P\left(\sup_{t\in [0,T]}\|B\,W_t\|_V\le\delta\right)>0.
\end{equation}
\end{hypothesis}

\begin{remark}
If $V$ is a separable Hilbert space, Hypothesis~\ref{hyp:reg2} follows from
\begin{equation}
B\in L_2(U,V).
\end{equation}
This can be easily verified by standard properties of Brownian motion by choosing a cylindrical representation.
\end{remark}

\subsection{Hypotheses for ergodicity and mixing}

\begin{hypothesis}[Ergodicity]\label{hyp:quant}
Assume that $0\le\beta\le\alpha-2$. Assume that $\eta\equiv 0$ or $\rho\equiv 0$ (locally monotone case). Furthermore, for $\lambda_0\in (0,1)$ and $\lambda_i\in (0,1)$, $i=1,2,3$, such that $\sum_{i=0}^3\lambda_i=1$, and for
\begin{equation}\label{eq:quant00}
\begin{split}
&c_1:=\frac{\alpha(\beta+2)}{2(\alpha+\beta)\left(\lambda_1\delta_1 c_0^\alpha\frac{\alpha+\beta}{\beta}\right)^{\frac{\beta}{\alpha}}}\|B\|_{L_2(U,H)}^{\frac{2(\alpha+\beta)}{\alpha}},\\
&c_2:=\beta^{\frac{\alpha+\beta}{\alpha}}\frac{\alpha(\beta+2)}{2(\alpha+\beta)\left(\lambda_2\delta_1 c_0^\alpha\frac{\alpha+\beta}{\beta}\right)^{\frac{\beta}{\alpha}}}\|B\|_{L(U,H)}^{\frac{2(\alpha+\beta)}{\alpha}},\\
&c_3:=\frac{(\beta+2)(\alpha-\beta-2)}{2(\alpha+\beta)\left(\lambda_3\delta_1 c_0^\alpha\frac{\alpha+\beta}{2\beta+2}\right)^{\frac{2\beta+2}{\alpha-\beta-2}}}(\beta+2)^{\frac{\alpha+\beta}{\alpha-\beta-2}}\|B\|_{L_2(U,H)}^{\frac{2(\alpha+\beta)}{\alpha-\beta-2}},
\end{split}
\end{equation}
we assume that
\begin{equation}
\frac{2(c_1+c_2+c_3)C_2}{\lambda_0\delta_1(\beta+2)}\le\delta_2,
\end{equation}
where $\delta_2>0$ and $C_2\ge 0$ are as in Hypothesis \ref{hyp:monotonicity}.
Furthermore, if $\alpha=\beta+2$, then $c_3=0$ and we assume additionally that
\begin{equation}
\|B\|^2_{L_2(U,H)}\le\lambda_3\frac{1}{\alpha}\delta_1 c_0^\alpha.
\end{equation}
\end{hypothesis}

\begin{hypothesis}[Exponential ergodicity and mixing]\label{hyp:mixing}
Assume that $0\le\beta\le\alpha-2$. Assume that $\eta\equiv 0$ or $\rho\equiv 0$ (locally monotone case). In the situation of Hypothesis \ref{hyp:quant}, assume that there exists $\gamma\in (0,\delta_2]$, such that
\begin{equation}
\frac{2(c_1+c_2+c_3)C_2}{\lambda_0\delta_1(\beta+2)}\le\delta_2-\gamma.
\end{equation}
\end{hypothesis}

\begin{remark}~
If $\beta=0$, then $\lambda_1=\lambda_2=0$ and
\begin{align*}
&c_1=\|B\|^2_{L_2(U,H)},\\
&c_2=0.
\end{align*}
For $\alpha=2$ and $\beta=0$,
\begin{align*}
&c_1=\|B\|^2_{L_2(U,H)},\\
&c_2=0\\
&c_3=0,
\end{align*}
and Hypothesis \ref{hyp:quant} simplifies to $\lambda_0=1-\lambda_3$,
\begin{equation}
\frac{2\|B\|^2_{L_2(U,H)}C_2}{\lambda_0\delta_1}\le\delta_2,
\end{equation}
and
\begin{equation}
\|B\|^2_{L_2(U,H)}\le\frac{1}{2}\lambda_3\delta_1 c_0^2.
\end{equation}
Accordingly, Hypothesis \ref{hyp:mixing} simplifies to $\gamma\in (0,\delta_2]$,
\begin{equation}
\frac{2\|B\|^2_{L_2(U,H)}C_2}{\lambda_0\delta_1}\le\delta_2-\gamma,
\end{equation}
and
\begin{equation}
\|B\|^2_{L_2(U,H)}\le\frac{1}{2}\lambda_3\delta_1 c_0^2.
\end{equation}
\end{remark}

\subsection{Existence and uniqueness of solutions}

We start recalling the definition of solution for~\eqref{eq:model} that  we shall use here.

\begin{definition}\label{def:solution}
An $H$-valued continuous $(\Fcal_t)_{t\in [0,T]}$-adapted process $(X_t)_{t\in [0,T]}$ is called a \emph{solution} to~\eqref{eq:model} if for its $\ud t\otimes\P$-equivalence class $\widehat{X}$, we have
\begin{enumerate}
\item For some $\tilde{\alpha}>1$, $\widehat{X}\in L^{\tilde{\alpha}}([0,T];V)\cap L^2([0,T];H)$, $\P$-a.s.
\item For any $V$-valued $(\Fcal_t)_{t\in [0,T]}$-progressively measurable $\ud t\otimes \P$-version $\overline{X}$ of $\widehat{X}$ it holds that
\begin{equation*}
X_t=x+\int_0^t A(\overline{X}_s)\,\ud s+B \,W_t,\quad t\in [0,T],\quad\P\text{-a.s.}
\end{equation*}
\end{enumerate}
\end{definition}

Now, we recall the existence and uniqueness result given in~\cite{LR:15}. Compare with \cite{BLZ:14} for L\'evy noise. Compare also~\cite{AV:24} for a novel alternative approach to locally monotone SPDEs using
interpolation spaces.

\begin{lemma}[Existence and uniqueness (locally monotone case)]
\label{lem:wellposed}
Let $T>0$ be fixed. 
Assume that Hypotheses \ref{hyp:hemicont}, \ref{hyp:coerc-g}, \ref{hyp:monotonicity-g}, \ref{hyp:growth} and \ref{hyp:regularity} hold true. Assume that $\eta\equiv 0$ or $\rho\equiv 0$ (locally monotone case). Then for any $x\in L^{\hat{\beta}+2}(\Omega,\Fcal_0,\P;H)$, there exists a unique solution to~\eqref{eq:model} in the sense of Definition~\ref{def:solution} with $\tilde{\alpha}=\hat{\alpha}$, where $\hat{\alpha}$ is as in Hypothesis \ref{hyp:coerc-g} and $\hat{\beta}$ is as in Hypothesis \ref{hyp:monotonicity-g}.
\end{lemma}
\begin{proof}
By our hypotheses, the conditions 
of~\cite[Theorem~1.2]{LR:15} are satisfied.
\end{proof}

For the fully locally monotone case, we obtain probabilistically weak solutions, see \cite{RSZ:24,KM:24}.

\begin{lemma}[Existence and uniqueness (fully locally monotone case)]
\label{lem:wellposed2}
Let $T>0$ be fixed. 
Assume that Hypotheses \ref{hyp:hemicont}, \ref{hyp:coerc-g}, \ref{hyp:monotonicity-g}, \ref{hyp:growth} and \ref{hyp:regularity} hold true. Assume that both $\eta$ and $\rho$ are non-trivial (fully locally monotone case). Then there exists a stochastic basis $(\tilde{\Omega},\tilde{\Fcal},\{\tilde{\Fcal}_t\}_{t\ge 0},\tilde{\P},\{\tilde{W}_t\}_{t\ge 0})$ such that for any $x\in H$, there exists a unique probabilistically weak solution to~\eqref{eq:model} in the sense of the statements of Definition~\ref{def:solution} holding with respect to this stochastic basis for $\tilde{\alpha}=\hat{\alpha}$, where $\hat{\alpha}$ is as in Hypothesis \ref{hyp:coerc-g}.
\end{lemma}
\begin{proof}
By our hypotheses, the conditions 
of~\cite[Theorem~2.6]{RSZ:24} are satisfied.
\end{proof}
Note that in the latter case, Hypothesis \ref{hyp:monotonicity-g} implies that the embedding $V\subset H$ is compact.

\subsection{Invariant measures}\label{subsec:ergo}

Recall that the Markovian semigroup $(P_t)_{t\ge 0}$ associated to~\eqref{eq:model} acts as follows
\[
P_t F(x):=\E[F(X_t^x)]\quad 
\textrm{ for any }\quad F\in B_b(H)\quad \textrm{ and }\quad x\in H,
\]
where $B_b(H):=\{F:H\to \mathbb{R}:\,F\textrm{ is bounded and Borel measurable}\}$. See~\cite[Proposition~4.3.5]{LR:15} for a proof of the Markov property in the Gaussian noise case. See also~\cite[Section~6.4]{GT:14} for a discussion of the Markov property for additive noise SPDEs.
 For a semigroup $(P_t)_{t\ge 0}$, we define the dual semigroup $(P_t^\ast)_{t\ge 0}$ acting on $\Mcal_1(H,\Bcal(H)):=\{\mu:\Bcal(H)\to [0,1]:\,\mu \textrm{ is a probability measure}\}$ by
\[P_t^\ast \mu(A):=\int_H P_t \mathbbm{1}_A(x)\,\mu(\ud x)\quad \textrm{ for any }\quad A\in\Bcal(H),
\]
where $\Bcal(H)$ denote the Borel sets of $H$ and $\mathbbm{1}_A$ denotes the indicator function of the set $A$, see~\cite{DPZ:96} for details.

\begin{definition}[Stochastically continuous Feller semigroup]
We say that $(P_t)_{t\ge 0}$ is a stochastically continuous Feller semigroup if for every $x\in H$ and every $r>0$ it follows that
\[\lim_{t\to 0} P_t^\ast\delta_x (B(x,r))=1
\quad \textrm{ and }\quad
P_t(C_b(H))\subset C_b(H),
\]
where $B(x,r):=\{y\in H:\|y\|_H<r\}$,
$\delta_x$ denotes the Dirac delta measure at $x$
and 
 $C_b(H):=\{F:H\to \mathbb{R}:\,F\textrm{ is continuous and bounded}\}$.
\end{definition}

\begin{definition}[$e$-property]\label{def:e}
We say that the semigroup $(P_t)_{t\ge 0}$ satisfies the \emph{$e$-property} if for every 
$F\in \Lip_b(H)$, for every $x\in H$, and every $\eps>0$, there exists $\delta>0$ such that 
\[
|P_t F(x)-P_t F(y)|<\eps\quad \textrm{ for all }\quad y\in B(x,\delta)\quad \textrm{ and }\quad t\geq 0,
\]
where 
$\Lip_b(H):=\{F:H\to \mathbb{R}:\,F\textrm{ is Lipschitz and bounded}\}$.
\end{definition}

Note that, under certain conditions, the $e$-property can be derived from the \emph{eventual $e$-property}~\cite{LL:24}.

\begin{definition}[Invariant measure]
A measure $\mu_\ast\in\Mcal_1(H,\Bcal(H))$ is said to be \emph{invariant} for the semigroup $(P_t)_{t\ge 0}$ if $P^\ast_t\mu_\ast=\mu_\ast$ for all $t\ge 0$.
\end{definition}

We recall the following concepts defined e.g. in~\cite{KPS:10}.
\begin{definition}[Weak$^\ast$-mean ergodicity, weak law of large numbers]
\label{def:KPS}
We say that $\{P_{t}\}_{t\ge0}$ is \emph{weak$^\ast$-mean
ergodic} if there exists a Borel probability measure $\mu_\ast$ on $\Bcal(H)$, such that
\begin{equation}\label{weakstarergodic}\text{w-}\lim_{T\to\infty}\frac{1}{T}\int_0^T P_t^\ast\mu\,\ud t=\mu_\ast\quad\text{for every }\quad
\mu\in\Mcal_1(H,\Bcal(H)),
\end{equation}
where the limit is in the sense of weak convergence of probability measures.

We say that the \emph{weak law of large numbers} holds for $\{P_{t}\}_{t\ge0}$, for a function $F\in\operatorname{Lip}_{b}(H)$ and for a probability measure $\mu$ on $\Bcal(H)$ if
\[\P_\mu\text{-}\lim_{T\to\infty}\frac{1}{T}\int_0^T F(X_t^\mu)\,\ud t=\int_H F\,\ud\mu_\ast,\]
where $\mu_\ast$ denotes the invariant probability measure of $\{P_{t}\}_{t\ge 0}$ and $\{X_t^\mu\}_{t\ge 0}$ denotes the Markov process related to $\{P_{t}\}_{t\ge 0}$ whose initial distribution is $\mu$ and whose path measure is $\P_\mu$, and where the convergence takes place in $\P_\mu$-probability.
\end{definition}

As noted in~\cite[Remark~3]{KPS:10},~\eqref{weakstarergodic} implies uniqueness of the invariant probability measure.
Let us recall~\cite[Theorem~2]{KPS:10}, which is an extension of the lower bound technique by Lasota and Szarek~\cite{LS:06}. 
Define
\[Q_T\mu:=\frac{1}{T}\int_0^T P^\ast_s \mu\,\ud s
\]
and write $Q_T(x,\cdot):=Q_T\delta_x$.

\begin{lemma}[Komorowski--Peszat--Szarek]\label{lem:KPS}
Assume that $(P_t)_{t\ge 0}$ has the $e$-property and that there exists $z\in H$ such that for every bounded set $J$ and every $\delta>0$, we have
\begin{equation}\label{eq:KPS1}
\inf_{x\in J}\liminf_{T\to\infty} Q_T(x,B(z,\delta))>0.
\end{equation}
Suppose further that for every $\eps>0$ and every $x\in H$, there exists a bounded Borel set $K\subset H$ such that
\begin{equation}\label{eq:KPS2}
\liminf_{T\to\infty} Q_T(x,K)>1-\eps.
\end{equation}
Then there exists a unique invariant probability measure $\mu_\ast$ for $(P_t)_{t\ge 0}$ such that the semigroup $(P_t)_{t\ge 0}$ is weak$^\ast$-mean ergodic and the weak law of large numbers holds.
\end{lemma}
\begin{proof}
See~\cite[Theorem~2]{KPS:10}.
\end{proof}

\section{Examples}\label{sec:examples}

Let us start with two lemmas discussing Hypothesis~\ref{hyp:angle}.
If $A$ is of subgradient type on $V$, we may obtain Hypothesis~\ref{hyp:angle} more easily as follows.
We refer to~\cite{ET:99} for the terminology of the subgradient of a convex functional.

\begin{lemma}\label{lem:subgradient}
Suppose that $A=-\partial\Phi$, that is, $A$ is equal to the negative subgradient of a lower semi-continuous convex functional $\Phi:V\to \R$ such that $\inf_{u\in V}\Phi(u)>-\infty$. Then Hypothesis~\ref{hyp:angle} is satisfied.
\end{lemma}
\begin{proof}
As $V$ is reflexive, the proof of~\cite[Proposition~7.1]{GT:14} can be adapted verbatim.
\end{proof}
For operators with a subgradient principal part, Hypothesis~\ref{hyp:angle} can be checked as follows.

\begin{lemma}\label{lem:hypE}
If $A=A_0+F$ where $A_0=-\partial \Phi$ is the subgradient of a lower semi-continuous convex functional $\Phi:V\to \R$ such that $\inf_{u\in V}\Phi(u)>-\infty$, and $F:V\to V^\ast$ is a strongly measurable (nonlinear) operator.
Suppose that
for some constants $\kappa_1\in\R\setminus\{0\}$, $K_1\in\R$
that for every $u\in V$,
\begin{equation}\label{eq:F1}\|F (u)\|_{V^\ast}\le \kappa_1\langle A_0 u,u\rangle+K_1.\end{equation}
Furthermore, suppose that for some constants $K_2,\kappa_2\in\R$, for every $u\in V$, 
\begin{equation}\label{eq:F2}2\langle F(u),u\rangle\le\kappa_2\|F (u)\|_{V^\ast}+K_2,\end{equation}
or that
\begin{equation}\label{eq:F3}2\langle F(u),u\rangle\le\kappa_2\langle A_0 u,u\rangle+K_2,\end{equation}
where, in the second case, we additionally assume that $\kappa_2\in (-2,\infty)$, whenever $\kappa_1<0$.
Then $A$ satisfies Hypothesis~\ref{hyp:angle}.
\end{lemma}
\begin{proof}
By Lemma~\ref{lem:subgradient}, there exists $\tilde{\delta}_4>0$ and $\tilde{C}_4\in\R$ with
\begin{align*}2\langle A_0 u,u\rangle
\le \tilde{C}_4-\tilde{\delta}_4\|A_0 u\|_{V^\ast}
\end{align*}
for every $u\in V$. If assumption~\eqref{eq:F1}~and~\eqref{eq:F2} hold, and if $\kappa_1 > 0$, for $R>\kappa_2\vee 0$,
 \begin{align*}
 2\langle Au,u\rangle
 =&2\langle A_0 u,u\rangle+ 2\langle F(u),u\rangle\\
\le& \left(2+R\kappa_1\right)\langle A_0 u,u\rangle-R\kappa_1\langle A_0 u,u\rangle+\kappa_2\|F(u)\|_{V^\ast}+K_2\\
 \le &\left(1+\frac{R\kappa_1}{2}\right)\tilde{C}_4+RK_1+K_2-\left(1+\frac{R\kappa_1}{2}\right)\tilde{\delta}_4\|A_0 u\|_{V^\ast}-(R-\kappa_2)\|F(u)\|_{V^\ast}\\
\le &\left(1+\frac{R\kappa_1}{2}\right)\tilde{C}_4+RK_1+K_2-\left(\left(1+\frac{R\kappa_1}{2}\right)\tilde{\delta}_4\wedge (R-\kappa_2)\right)\|A_0 u+F(u)\|_{V^\ast}.
\end{align*}
Hence Hypothesis~\ref{hyp:angle} holds with $\delta_4=\left(1+\frac{R\kappa_1}{2}\right)\tilde{\delta}_4\wedge (R-\kappa_2)$ and $C_4=\left(1+\frac{R\kappa_1}{2}\right)\tilde{C}_4+RK_1+K_2$.

On the other hand, if $\kappa_1<0$,
 \begin{align*}
 2\langle Au,u\rangle
 =&2\langle A_0 u,u\rangle+ 2\langle F(u),u\rangle\\
\le&\langle A_0 u,u\rangle+\langle A_0 u,u\rangle+(\kappa_2\vee 0)\|F(u)\|_{V^\ast}+K_2\\
\le&\langle A_0 u,u\rangle+(1+(\kappa_2\vee 0))\langle A_0 u,u\rangle+K_1+K_2\\
 \le &\frac{\tilde{C}_4}{2}+K_1+K_2-\frac{\tilde{\delta}_4}{2}\|A_0 u\|_{V^\ast}-\frac{1+(\kappa_2\vee 0)}{|\kappa_1|}\|F(u)\|_{V^\ast}+(K_1\vee 0)\frac{1+(\kappa_2\vee 0)}{|\kappa_1|}\\
\le &\frac{\tilde{C}_4}{2}+K_2+(K_1\vee 0)\frac{1+(\kappa_2\vee 0)+|\kappa_1|}{|\kappa_1|}-\left(\frac{\tilde{\delta}_4}{2}\wedge \frac{1+(\kappa_2\vee 0)}{|\kappa_1|}\right)\|A_0 u+F(u)\|_{V^\ast}.
\end{align*}
Hence Hypothesis~\ref{hyp:angle} holds with $\delta_4=\frac{\tilde{\delta}_4}{2}\wedge \frac{1+(\kappa_2\vee 0)}{|\kappa_1|}$ and $C_4=\frac{\tilde{C}_4}{2}+K_2+(K_1\vee 0)\frac{1+(\kappa_2\vee 0)+|\kappa_1|}{|\kappa_1|}$.

If assumptions~\eqref{eq:F1}~and~\eqref{eq:F3} hold,
if $\kappa_1 > 0$, for $R>\frac{|\kappa_2|}{\kappa_1}$,
 \begin{align*}
 2\langle Au,u\rangle
 =&2\langle A_0 u,u\rangle+ 2\langle F(u),u\rangle\\
\le& \left(2+R\kappa_1+\kappa_2\right)\langle A_0 u,u\rangle-R\kappa_1\langle A_0 u,u\rangle+K_2\\
 \le &\left(1+\frac{R\kappa_1+\kappa_2}{2}\right)\tilde{C}_4+RK_1+K_2-\left(1+\frac{R\kappa_1+\kappa_2}{2}\right)\tilde{\delta}_4\|A_0 u\|_{V^\ast}-R\|F(u)\|_{V^\ast}\\
\le &\left(1+\frac{R\kappa_1+\kappa_2}{2}\right)\tilde{C}_4+RK_1+K_2-\left(\left(1+\frac{R\kappa_1+\kappa_2}{2}\right)\tilde{\delta}_4\wedge R\right)\|A_0 u+F(u)\|_{V^\ast}.
\end{align*}
Hence Hypothesis~\ref{hyp:angle} holds with $\delta_4=\left(1+\frac{R\kappa_1+\kappa_2}{2}\right)\tilde{\delta}_4\wedge R$ and $C_4=\left(1+\frac{\kappa_1+\kappa_2}{2}\right)\tilde{C}_4+RK_1+K_2$.

On the other hand, if $\kappa_1<0$, $\kappa_2\in (-2,0)$, for $\eps\in (0,2+\kappa_2)$,
 \begin{align*}
 2\langle Au,u\rangle
 =&2\langle A_0 u,u\rangle+ 2\langle F(u),u\rangle\\
\le&(2+\kappa_2-\eps)\langle A_0 u,u\rangle+\eps\langle A_0 u,u\rangle+K_2\\
\le&\frac{2+\kappa_2-\eps}{2}\tilde{C}_4-\frac{\tilde{\delta}_4(2+\kappa_2-\eps)}{2}\|A_0 u\|_{V^\ast}-\frac{\eps}{|\kappa_1|}\|F(u)\|_{V^\ast}+(K_1\vee 0)\frac{\eps}{|
\kappa_1|}+K_2\\
\le &\frac{2+\kappa_2-\eps}{2}\tilde{C}_4+(K_1\vee 0)\frac{\eps}{|
\kappa_1|}+K_2-\left(\frac{\tilde{\delta}_4(2+\kappa_2-\eps)}{2}\wedge \frac{\eps}{|\kappa_1|}\right)\|A_0 u+F(u)\|_{V^\ast}.
\end{align*}
Hence Hypothesis~\ref{hyp:angle} holds with $\delta_4=\frac{\tilde{\delta}_4(2+\kappa_2-\eps)}{2}\wedge \frac{\eps}{|\kappa_1|}$ and $C_4=\frac{2+\kappa_2-\eps}{2}\tilde{C}_4+(K_1\vee 0)\frac{\eps}{|
\kappa_1|}+K_2$.
\end{proof}

\begin{remark}\label{rem:torus}
We note that the results of this section also hold in an analog manner if one replaces the Poincar\'e domain $\Ocal\subset\R^d$ with 
Dirichlet boundary conditions with the flat torus $\mathbb{T}^d:=\R^d/(2\pi\Z^d)$, that is, periodic boundary conditions, combined with the requirement that $\int_{\T^d}X_t\,\ud x=0$, $\P$-a.e. for every $t\ge 0$.
\end{remark}

\subsection{Stochastic heat equation}\label{subsec:ex1}

Let us prove Theorem~\ref{thm:heat}.
Consider the stochastic heat equation for $\nu>0$ on a Poincar\'e domain $\Ocal\subset\R^d$ with Lipschitz boundary,
\begin{equation}\label{eq:heat-}\ud X_t=\nu\Delta X_t\,\ud t+B\,\ud W_t,\quad X_0=x,\end{equation}
where $x\in H:=L^2(\Ocal)$. We impose Dirichlet boundary conditions, so that $V:=W^{1,2}_0(\Ocal)$.
Due to Poincar\'e's inequality \cite{AF:03}, we shall equip $V$ with the norm
\[\|v\|_V:=\left(\int_\Ocal |\nabla v|^2\,\ud x\right)^{1/2},\quad v\in V.\]
Note that the embedding constant $c_0$ is the inverse of the Poincar\'e constant of $\Ocal$. 
\begin{description}
\item[Hypothesis~\ref{hyp:hemicont}] The hemicontinuity has been proved in~\cite[Example~4.1.7]{PR:07}.
\item[Hypotheses \ref{hyp:coerc} and \ref{hyp:coerc-g}] The coercivity has been proved in~\cite[Example~4.1.7]{PR:07}, with $\alpha=\hat{\alpha}=2$ and $\delta_1=2\nu$.
\item[Hypothesis~\ref{hyp:monotonicity} and \ref{hyp:monotonicity-g}] The monotonicity has been proved in~\cite[Example~4.1.7]{PR:07} with $\alpha=\hat{\alpha}=2$, $\beta=\hat{\beta}=0$, $\delta_2=2\nu c_0^2$, $\rho\equiv 0$, and $C_2=0$.
\item[Hypothesis~\ref{hyp:growth}] The growth condition has been proved in~\cite[Example~4.1.7]{PR:07}.
\item[Hypothesis~\ref{hyp:angle}] Hypothesis~\ref{hyp:angle} follows from Lemma~\ref{lem:subgradient} for the functional $\Phi=\frac{1}{2}\|\cdot\|_{V}^2$.
\end{description}
Hence, Theorem~\ref{thm:heat} follows immediately from Theorem~\ref{thm:main} and Theorem~\ref{thm:mixing}.

\subsection{Semilinear stochastic equations}\label{subsec:ex2}

Let us prove Theorem~\ref{thm:semi}.
Consider the stochastic semilinear equation on a Poincar\'e domain $\Ocal\subset\R^d$, $d\in\N$, $\nu>0$, with Lipschitz boundary,
\begin{equation}\label{eq:semi1-}\ud X_t=\nu\Delta X_t\,\ud t+\sum_{i=1}^d f_i(X_t)\partial_i X_t\,\ud t+g(X_t)\,\ud t+B\,\ud W_t\quad X_0=x,\end{equation}
where $x\in H:=L^2(\Ocal)$.
We impose Dirichlet boundary conditions, so that $V:=W^{1,2}_0(\Ocal)$.
Due to Poincar\'e's inequality \cite{AF:03}, we shall equip $V$ with the norm
\[\|v\|_V:=\left(\int_\Ocal |\nabla v|^2\,\ud x\right)^{1/2},\quad v\in V.\]
Note that the embedding constant $c_0$ is the inverse of the Poincar\'e constant of $\Ocal$.
We assume that $f_1(x)=x$ in $d=1$ or that $f_i\in\Lip_b(\R)$, $i=1,\ldots,d$ for $d=1,2$.
Denote $\mathbf{f}:=(f_1,\ldots,f_d)^{\textup{t}}$.
We assume that $g:\R\to\R$ is continuous with $g(0)=0$, and that there exist $C,c\ge 0$, $s\in [0,2]$, such that
\begin{equation}\label{eq:g1-}|g(x)|\le C\left(1+|x|^2\right),\quad x\in\R,\end{equation}
and
\begin{equation}\label{eq:g2-}\left(g(x)-g(y)\right)(x-y)\le c\left(1+|y|^s\right)(x-y)^2,\quad x,y\in\R.\end{equation}
For $d=1$, assume that $c<\nu c_0^2$. For $d=2$, assume that
\begin{equation}\label{eq:g5-}\frac{c}{2\nu c_0^2}+\frac{\|\mathbf{f}\|_{L^\infty}}{\nu c_0}<\frac{1}{2}.\end{equation}
Let us also assume that there exist $K,k\ge 0$ with
\begin{equation}\label{eq:g3-}g(x)x\le K+ k|x|^2,\quad x\in\R.\end{equation}
Assume that
\begin{equation}\label{eq:g4-}\frac{\|\mathbf{f}\|_{L^\infty}}{\nu c_0}+\frac{k}{\nu c_0^2}<2.\end{equation}
Note that assumption~\eqref{eq:g3-}~and~\eqref{eq:g4-} are needed to prove Hypothesis~\ref{hyp:angle}. 
\begin{description}
\item[Hypothesis~\ref{hyp:hemicont}] The hemicontinuity can be proved as in~\cite[Lemma~5.1.6 and Example~5.1.7]{LR:15}.
\item[Hypothesis~\ref{hyp:coerc} and \ref{hyp:coerc-g}] The coercivity can be proved as in~\cite[Lemma~5.1.6 and Example~5.1.7]{LR:15} with $\alpha=\hat{\alpha}=2$.
\item[Hypothesis~\ref{hyp:monotonicity} and \ref{hyp:monotonicity-g}] The monotonicity for $\alpha=\hat{\alpha}=2$ can be proved as in~\cite[Lemma~5.1.6 and Example 5.1.7]{LR:15}. As we have to give some attention to the constant $\delta_2>0$, we will discuss part of the arguments here. Let $d=1$ and $c<\nu c_0^2$. Then, 
by~\cite[Equation~(5.15) and Example~5.1.8]{LR:15},
\[\begin{split}
&2\langle A(u)-A(v),u-v\rangle \\
\le& -2\nu\|u-v\|_V^2+2\Lip(\mathbf{f})\left(\|u-v\|_V\|v\|_{L^\infty}\|u-v\|_{H}+\|v\|_V\|u-v\|_{L^4}^2\right)\\
&+2c\left(1+\|v\|_{L^\infty}^s\right)\|u-v\|_H^2.\end{split}\]Now, due to
\begin{equation}\label{eq:interpol}\|u\|_{L^{4}}^2\le 2\|u\|_{L^2}\|\nabla u\|_{L^2}\end{equation}
for every $u\in W_0^{1,2}(\Ocal)$ in $d=1,2$, we get that for any $\eps\in (0,2\nu c_0^2-2c)$, there exists a constant $C=C(d,\Ocal,c_0,\operatorname{Lip}(\mathbf{f}),\eps,c,s)>0$ with
\[2\langle A(u)-A(v),u-v\rangle \le -(2\nu c_0^2-2c-\eps)\|u-v\|_H^2+C\|v\|_V^2\|u-v\|_H^2.\]

Let $d=2$. Then, by~\cite[Equations~(5.15)~and~(5.16), and Example~5.1.8]{LR:15},
\[\begin{split}
&2\langle A(u)-A(v),u-v\rangle \\
\le& -2\nu\|u-v\|_V^2+4\|\mathbf{f}\|_{\infty}\|u-v\|_V\|u-v\|_{H}+2\Lip(\mathbf{f})\|v\|_V\|u-v\|_{L^4}^2\\
&+2c\left(1+\|v\|_{L^{2s}}^s\right)\|u-v\|_H^2\\
\le& -\left(2\nu c_0^2-4c_0\|\mathbf{f}\|_\infty\right)\|u-v\|_H^2+2\Lip(\mathbf{f})\|v\|_V\|u-v\|_{L^4}^2\\
&+2c\left(1+\|v\|_{L^{2s}}^s\right)\|u-v\|_H^2.\end{split}\]
Now, by \eqref{eq:g5-} and by~\eqref{eq:interpol}, we get that for any $\eps\in (0,2\nu c_0^2-2c-4c_0\|\mathbf{f}\|_\infty)$, there exists a constant $C=C(d,\Ocal,c_0,\operatorname{Lip}(\mathbf{f}),\eps,c,s)>0$ with
\[\begin{split}
&2\langle A(u)-A(v),u-v\rangle \\
\le& -\left(2\nu c_0^2-4c_0\|\mathbf{f}\|_\infty-\eps\right)\|u-v\|_H^2+C\|v\|_{V}^2\|u-v\|_{H}^2.\end{split}\]
Hence, in $d=1,2$, Hypothesis \ref{hyp:monotonicity} holds with $\beta=\hat{\beta}=0$ and $\alpha=2$.
\item[Hypothesis~\ref{hyp:growth}] Hypothesis~\ref{hyp:growth} follows as 
in~\cite[Lemma~5.1.6 and Example~5.1.7]{LR:15}.
\item[Hypothesis~\ref{hyp:angle}] Note that $\nu\Delta u=-\partial\Phi(u)$ for the functional $\Phi=\frac{\nu}{2}\|\cdot\|_{V}^2$.
As $\mathbf{f}$ is essentially bounded, we get that
\[\langle (\mathbf{f}(u)\cdot\nabla u), u\rangle\le \|\mathbf{f}\|_{L^\infty}\|u\|_V\|u\|_H\le \frac{\|\mathbf{f}\|_{L^\infty}}{c_0}\|u\|_V^2=-\frac{\|\mathbf{f}\|_{L^\infty}}{\nu c_0}\langle \nu\Delta u,u\rangle.\]
If $d=1$ and $f_1(x)=x$, then by integration by parts,
\[\langle (\mathbf{f}(u)\cdot\nabla u), u\rangle=\langle u\partial_x u, u\rangle=-2\langle u\partial_x u, u\rangle=0.\]
Furthermore,
\begin{align*}&\|\mathbf{f}(u)\cdot\nabla u\|_{V^\ast}\le\frac{1}{c_0}\|\mathbf{f}(u)\cdot\nabla u\|_{H}\\
\le&\frac{\|\mathbf{f}\|_{L^\infty}}{c_0}\|u\|_V\le\frac{\|\mathbf{f}\|_{L^\infty}}{c_0}\|u\|_V^2+\frac{\|\mathbf{f}\|_{L^\infty}}{c_0}\\
=&-\frac{\|\mathbf{f}\|_{L^\infty}}{\nu c_0}\langle \nu\Delta u,u\rangle+\frac{\|\mathbf{f}\|_{L^\infty}}{c_0}.\end{align*}
By~\eqref{eq:g3-}, we have that
\[\langle g(u),u\rangle\le k\|u\|^2_{L^2}+K\le \frac{k}{c_0^2}\|u\|^{2}_V+K.\]
Hence,
\[\langle g(u),u\rangle \le-\frac{k}{\nu c_0^2} \langle \nu\Delta u,u\rangle+K.\]
Similarly, by the Sobolev embedding theorem, there exists $c_4>0$ with
\[\|g(u)\|_{V^\ast}\le\frac{1}{c_0}\|g(u)\|_{L^2}\le\frac{C}{c_0}\|u\|_{L^4}^2+C\le\frac{c_4 C}{c_0}\|u\|_{V}^2+C=-\frac{c_4 C}{\nu c_0} \langle \nu\Delta u,u\rangle+
C.\]
Now by~\eqref{eq:g4-}, Hypothesis~\ref{hyp:angle} follows now from Lemma~\ref{lem:hypE}.
\end{description}
Consequently, Theorem~\ref{thm:semi} follows directly from Theorem~\ref{thm:main} and Theorem~\ref{thm:mixing}.

\subsection{Stochastic incompressible 2D Navier-Stokes equations}\label{subsec:ex3}

Let us prove Theorem~\ref{thm:2DNSE}.
The deterministic incompressible 2D Navier-Stokes equations can be formulated as
\begin{align*}
\partial_t u(t)&=\nu\Delta u(t)-(u(t)\cdot\nabla) u(t)-\nabla p(t)+f(t),\\
\nabla\cdot u(t)&=0,
\end{align*}
where $u:[0,T]\times\Ocal\to\R^2$ denotes the velocity field of an incompressible Newtonian fluid with viscosity $\nu>0$, $p:[0,T]\times\Ocal\to\R$ denotes the pressure, and $f:[0,T]\times\Ocal\to\R^2$ denotes an external force,
where $\Ocal\subset\R^2$ is a Poincar\'e domain with sufficiently smooth boundary, see~\cite{T:01,AF:03}. We employ no-slip (that is, Dirichlet) boundary conditions. Define
\[V:=\left\{v\in W^{1,2}_0(\Ocal;\R^2)\;\colon\;\nabla \cdot v=0\text{ a.e. in }\Ocal\right\}.\]
Due to Poincar\'e's inequality, we shall equip $V$ with the norm
\[\|v\|_V:=\left(\int_\Ocal |\nabla v|^2\,\ud x\right)^{1/2},\quad v\in V.\]
We define $H$ as the closure of $V$ in $L^2(\Ocal;\R^2)$ w.r.t. the standard $L^2$-norm. $H$ is a closed Hilbert subspace of $L^2(\Ocal;\R^2)$, so
the Helmholtz-Leray projection
\[\mathbf{P}:L^2(\Ocal;\R^2)\to H\]
is given by the orthogonal projection. The Stokes operator with viscosity constant $\nu$ is given by
\[A_0:W^{2,2}(\Ocal;\R^2)\cap V\to H,\quad A_0 u=\nu \mathbf{P} \Delta u.\]
Denote $F(u):=F(u,u)$, where
\[F(u,v)=-\mathbf{P}[(u\cdot\nabla)v],\quad u,v\in V.\]
Denote the extensions $A_0:V\to V^\ast$ and $F:V\times V\to V^\ast$ by the same respective symbols. Note that these extensions can be defined by duality for Lipschitz boundary.
Consider the \emph{stochastic incompressible 2D Navier-Stokes equations} with viscosity $\nu>0$ driven by degenerate additive Wiener noise
\begin{equation}\label{eq:2DNSE-}
\ud X_t=(A_0 X_t+F(X_t))\,\ud t+B\,\ud W_t,\quad X_0=x,
\end{equation}
where $B\in L_2(U,V)$.
Existence and uniqueness of solutions to~\eqref{eq:2DNSE} have been discussed in~\cite[Example~2.6]{BLZ:14}. Let us verify our Hypotheses in the Gelfand triple
\[
V\subset H\equiv H^\ast \subset V^\ast.
\]
Let us recollect Ladyzhenskaya's inequality, see~\cite{L:59,T:01},
\[\langle F(u,v),u\rangle \le 2\|u\|_H  \|u\|_V\|v\|_V,\quad u,v\in V.\]
Note that the embedding constant $c_0$ is the inverse of the Poincar\'e constant of $\Ocal$. 
\begin{description}
\item[Hypothesis~\ref{hyp:hemicont}] The hemicontinuity has been proved in~\cite[Example~2.6]{BLZ:14}.
\item[Hypothesis~\ref{hyp:coerc} and \ref{hyp:coerc-g}] The coercivity has been 
proved in \cite[Example~2.6]{BLZ:14}, with $\alpha=\hat{\alpha}=2$ and $\delta_1=2\nu$.
\item[Hypothesis~\ref{hyp:monotonicity}] Noting that by integration by parts and the incompressibility condition,
$\langle F(u,v),v\rangle=0$ for any $u,v\in V$, we get by Ladyzhenskaya's and Young's inequalities, respectively, for all $u,v\in V$,
\begin{align*}
&2\langle F(u)-F(v),u-v\rangle\\
=&2\langle F(u,u-v),u-v\rangle +2\langle F(u-v,v),u-v\rangle\\
=&2\langle F(u-v,v),u-v\rangle\\
\le&4\|u-v\|_H\|u-v\|_V\|v\|_V\\
\le&\nu\|u-v\|_V^2+\frac{4}{\nu}\|v\|_{V}^2\|u-v\|_H^2.
\end{align*}
Hence
\[2\langle A_0 u+F(u)-A_0 v-F(v),u-v\rangle\le-\nu\|u-v\|^2_V+\frac{4}{\nu}\|v\|_V^2\|u-v\|_H^2.\]
We get that Hypothesis~\ref{hyp:monotonicity} holds with $\alpha=2$, $\beta=0$, $\delta_2=\nu c_0^2$ and $\rho(v)=\frac{4}{\nu}\|v\|_{V}^2$ and $C_2=\frac{4}{\nu}$.
\item[Hypothesis~\ref{hyp:monotonicity-g}] Hypothesis \ref{hyp:monotonicity-g} has been proved in \cite[Example~2.6]{BLZ:14} for $\rho(v)=\frac{64}{\nu^3}\|v\|_{L^4(\Ocal;\R^2)}^4$, $K_2=-\nu c_0^2$, $\hat{\alpha}=2$ and $\hat{\beta}=2$. 
\item[Hypothesis~\ref{hyp:growth}] Hypothesis~\ref{hyp:growth} holds 
by~\cite[Example~2.6]{BLZ:14} for $\hat{\alpha}=2$ and $\hat{\beta}=2$.
\item[Hypothesis~\ref{hyp:angle}]
Note that $A_0$ has the convex potential $\frac{1}{2}\|A_0^{1/2} u\|_{H}^2$, $u\in V$.
Furthermore, we have that $\langle F(u),u\rangle =0$ for every $u\in V$. By H\"older's inequality,
\[|\langle F(u),v\rangle|\le \|u\|_{L^{4}(\Ocal;\R^2)}^2\|v\|_V,\quad u,v\in V,\]
which yields by the Sobolev embedding $V\subset L^4(\Ocal;\R^2)$ for some constant $C>0$ that
\[\|F(u)\|_{V^\ast}\le C\|u\|_V^2=\frac{C}{\nu}\langle A_0^{1/2} u,A_0^{1/2} u\rangle = - \frac{C}{\nu}\langle A_0 u,u\rangle,\]
for every $u\in V$. 
Hypothesis~\ref{hyp:angle} follows now from Lemma~\ref{lem:hypE}.
\end{description}
Then, Theorem~\ref{thm:2DNSE} follows directly from Theorem~\ref{thm:main} and Theorem~\ref{thm:mixing}.

\begin{remark}\label{rem:detNSE}
The above example includes the case of the deterministic 2D Navier-Stokes equations.
We obtain that the unique invariant measure is a Dirac measure supported in $0\in H$, being the unique stationary solution $u_\infty$ to
\[
A_0 u_\infty+F(u_\infty)=0
\]
with Dirichlet boundary conditions. By taking $\langle\varphi,\cdot\rangle$, $\varphi\in H$ and $\|\cdot\|^2_H$ as test-functions for the semigroup, the mixing in $2$-Wasserstein distance yields strong convergence in $H$ to zero for the solutions to the deterministic 2D Navier-Stokes equations, as $t\to\infty$ (noting that, in Hilbert spaces, weak convergence, together with convergence of the norms, yields strong convergence). We recover the deterministic stability and extinction results in \cite[Chapter 10]{T:01}.
\end{remark}

\subsection{Stochastic shear thickening incompressible power-law fluids}\label{subsec:ex4}

Let us prove Theorem~\ref{thm:plaw}.
Consider the velocity field of a viscous and incompressible non-Newtonian fluid perturbed by Wiener noise with Dirichlet boundary conditions on a sufficiently smooth Poincar\'e domain $\Ocal\subset\R^d$, $d\in\N$, with outward unit normal $n$ on $\partial\Ocal$, $d\ge 2$. Let $p\ge 2$, $\nu>0$, and assume that $p\ge 1+\frac{d}{2}$. The case $p>2$ is called \emph{shear thickening}. For $u:\Ocal\to\R^d$, define
\[e(u):\Ocal\to\R^d\otimes\R^d,\quad e_{i,j}(u):=\frac{\partial_i u_j+\partial_j u_i}{2},\quad 1\le i,j\le d,\]
and
\[\tau(u):\Ocal\to\R^d\otimes\R^d,\quad \tau(u):=2\nu(1+|e(u)|)^{p-2}e(u).\]
Let
\[V:=\{u\in W_0^{1,p}(\Ocal;\R^d)\;\colon\;\operatorname{div}(u)=0\text{ a.e. in }\Ocal\},\]
and
\[H:=\{u\in L^2(\Ocal;\R^d)\;\colon\;\operatorname{div}(u)=0\text{ a.e. in }\Ocal,\;u\cdot n=0\text{ on }\partial\Ocal\}.\]
Denote the Helmholtz-Leray projection by $\mathbf{P}:L^2(\Ocal;\R^d)\to H$. Define the nonlinear $p$-Stokes operator by
\[A_0:W^{2,p}(\Ocal;\R^d)\cap V\to H,\quad A_0(u):=\mathbf{P}(\operatorname{div}(\tau(u))),\]
and define the convection term $F:(W^{2,p}(\Ocal;\R^d)\cap V)\times(W^{2,p}(\Ocal;\R^d)\cap V)\to H$ as before in the case of the 2D Navier-Stokes equations as
\[F(u,v):=-\mathbf{P}[(u\cdot\nabla)v],\quad F(u):=F(u,u).\]
Denote the extensions $A_0:V\to V^\ast$ and $F:V\times V\to V^\ast$ by the same respective symbols.
Note that these extensions can be defined by duality for Lipschitz boundary.
Consider the stochastic power-law fluid equations
\begin{equation}\label{eq:plaw-}
\ud X_t=(A_0 X_t+F(X_t))\,\ud t+B\,\ud W_t,\quad X_0=x.
\end{equation}
Existence and uniqueness of this equation has been discussed in~\cite[Example~2.9]{BLZ:14}. 
\begin{description}
\item[Hypothesis~\ref{hyp:hemicont}] The hemicontinuity has been proved in~\cite[Example~2.9]{BLZ:14}.
\item[Hypothesis~\ref{hyp:coerc} and \ref{hyp:coerc-g}] The coercivity has been 
proved~\cite[Example~2.9]{BLZ:14}, with $\alpha=\hat{\alpha}=p$ for some $\delta_1>0$.
\item[Hypothesis~\ref{hyp:monotonicity} and \ref{hyp:monotonicity-g}] By~\cite[Example~2.9]{BLZ:14}, we get for all $u,v\in V$, that there exists $C>0$, such that for any $\eps>0$, there exists $C(\eps)>0$ with
\[2\langle A_0 u+F(u)-A_0 v-F(v),u-v\rangle\le-(C-\eps)\|u-v\|^2_V+\rho(v)\|u-v\|_H^2,\]
where $\rho(v):=C_\eps\|v\|_{V}^{\frac{2p}{2p-d}}$, and $\rho(v)\le C_\eps\|v\|_V^{p}\|v\|_H^{\beta}$ for $\beta=\hat{\beta}=\frac{2p}{2p-d}-p=\frac{p(2-2p+d)}{2p-d}$.
\item[Hypothesis~\ref{hyp:growth}] Hypothesis~\ref{hyp:growth} holds 
by~\cite[Example~2.9]{BLZ:14}.
\item[Hypothesis~\ref{hyp:angle}] Note that $A_0$ is the subgradient of the continuous convex\footnote{Note that convexity follows from the fact that $f(x):=2\nu(p^{-1}(1+|x|)^p-(p-1)^{-1}(1+|x|)^{p-1})$ has the non-negative second derivative $f''(x)=2\nu (|x|+1)^{p-3}(1+(p-1)|x|)$.} potential on $V$
\[v\mapsto 2\nu\int_{\Ocal} \left(\frac{1}{p}(1+|e(v)|)^p-\frac{1}{p-1}(1+|e(v)|)^{p-1}\right)\,\ud x,\quad v\in V.\]
By the Riesz-Thorin interpolation inequality and the Sobolev embedding theorem, for $q=\frac{dp}{d-p}$ and $\gamma=\frac{d}{(d+2)p-2d}$, for some constant $C>0$,
\[\|v\|_{L^{\frac{2p}{p-1}}}\le\|v\|^\gamma_{L^q}\|v\|^{1-\gamma}_H\le C\|v\|_V,\quad v\in V.\]
Note that by Korn's inequality~\cite[Theorem~1.10~(p.~196)]{MNRR:96} and the previous inequality, for some $C>0$, which might change from line to line,
\begin{equation}\label{eq:korn}\begin{split}
-\langle A_0 v,v\rangle
=& C+C\int_{\Ocal}|e(v)|^p\,\ud x\\
\ge& C+ C\|v\|_V^p\\
\ge& C+C\|v\|^p_{L^{\frac{2p}{p-1}}}\\
\ge& C+C\|v\|^2_{L^{\frac{2p}{p-1}}},\quad v\in V,\end{split}\end{equation}
as $p\ge 2$.
Also,
\begin{equation}\label{eq:Fast}\|F(v)\|_{V^\ast}\le \|v\|^2_{L^{\frac{2p}{p-1}}},\quad v\in V,\end{equation}
see~\cite[Example~2.9]{BLZ:14}.
Therefore, Hypothesis~\ref{hyp:angle} is satisfied by Lemma~\ref{lem:hypE}, noting that $\langle F(u),u\rangle =0$ for $u\in V$.
\end{description}
For Hypothesis~\ref{hyp:quant}, we need that, $\alpha\ge 2$, $\beta\ge 0$ and that $\alpha-\beta-2\ge 0$ which is equivalent to $2p\ge d$, $p\ge 2$, $2-2p+d\ge 0$, and $(2p-2)(2p-d)\ge 2p$ in this situation. In particular, $d\ge 2$ and $p\in [\frac{d}{2},1+\frac{d}{2}]$. Hence, as $p\ge 1+\frac{d}{2}$, we remain only with the case that $p=\alpha=1+\frac{d}{2}$, $\beta=0$. Note that $(2p-2)(2p-d)\ge 2p$ is automatically satisfied for $p=1+\frac{d}{2}$ and $d\ge 2$.
Now, Theorem~\ref{thm:plaw} follows from Theorem~\ref{thm:main} and Theorem~\ref{thm:mixing}.

\section{Proof of Theorem~\ref{thm:main}}\label{sec:proof}

In this section, we shall verify 
Theorem~\ref{thm:main} by proving that the conditions for Lemma~\ref{lem:KPS} hold.

\subsection{Main estimates}

In order to rigorously apply It\^o's formula, 
we construct approximating solutions to~\eqref{eq:model} via projections to finite dimensional subspaces using the Faedo-Galerkin's method.
In the sequel, we define the notation of the aforementioned projections.
Let $\{e_1,e_2,\ldots,\}\subset V$ be a complete orthonormal basis for $H$ such that $\operatorname{span}\{e_1,e_2,\ldots,\}$ is dense in $V$. 
For each $n\in \mathbb{N}$, we 
denote $H_n:=\operatorname{span}\{e_1,\ldots,e_n\}$ and set $Q_n:V^\ast\to H_n$ be the linear projection operator defined by 
\begin{equation}\label{eq:Qn}
Q_n y:=\sum_{j=1}^n \langle y,e_j\rangle e_j\quad \textrm{ for any }\quad  y\in V^\ast.\end{equation}
We note that $Q_n\vert_H$ is  the orthogonal projection from $H$ onto $H_n$, which, by abuse of notation, is denoted by the same symbol.

Let $\{g_1,g_2,\ldots,\}\subset U$ be a complete orthonormal basis for $U$.
We then denote for each $t\geq 0$
\[
W^{(n)}_t:=\tilde{Q}_n W_t:=\sum_{j=1}^n\langle W_t,g_j\rangle g_j.
\]
Bearing all this in mind,
for a fixed $x\in L^{\beta+2}(\Omega,\Fcal_0,\P;H)$ and for each $n\in\N$, we consider the solution $(X^{(n),x_n}_t)_{t\geq 0}$ of the following stochastic differential equation on the finite dimensional space $H_n$:
\begin{equation}\label{eq:finitedim0new}
\begin{split}
\ud X^{(n),x_n}_t&=Q_n A(X_t^{(n),x_n})\,\ud t+Q_n B\,\ud W^{(n)}_t,\quad t\geq 0,\\
X_0^{(n),x_n}&=Q_n x=:x_n.
\end{split}
\end{equation}
We point out that there exists a unique probabilistically strong solution to~\eqref{eq:finitedim0new}, see for instance~\cite{ABW:10,GK:80}.

In the locally monotone case, by~\cite[Proof of Theorem 5.1.3]{LR:15} or~\cite[Lemma~4.4]{BLZ:14}, there exists a subsequence $\{n_k\}$ such that
\begin{equation}\label{eq:Qconvergence}Q_{n_k} X^{(n_k),x_{n_k}}\rightharpoonup^\ast X^{x}\end{equation}
weakly$^\ast$ in $L^\infty([0,T];L^2(\Omega;H))$ and weakly in $L^{\hat{\alpha}}([0,T]\times\Omega,\Bcal\Fcal,\ud t\otimes\P;V)$ as $k\to\infty$, 
where 
\[
\Bcal\Fcal:=\{A\subset[0,T]\times\Omega\,\colon\,\forall t\in [0,T], A\cap ([0,t]\times \Omega)\in\Bcal([0,t])\otimes\Fcal_t\}
\] 
is the $\sigma$-field of progressively measurable sets on $[0,T]\times\Omega$ and $X^x$ is the solution to \eqref{eq:model} with initial datum $x\in H$.
In the fully locally monotone case, by~\cite{RSZ:24}, the convergence is $\tilde{\P}$-a.s. strongly in $L^{\hat{\alpha}}([0,T],H)$ and weakly in $L^{\hat{\alpha}}([0,T]\times\tilde{\Omega},\Bcal\tilde{\Fcal},\ud t\otimes\tilde{\P};V)$, where $(\tilde{\Omega},\tilde{\Fcal},\{\tilde{\Fcal}_t\}_{t\ge 0},\tilde{\P},\{\tilde{W}_t\}_{t\ge 0})$ is the stochastic basis chosen for the solution.

Let us prove the main a priori estimates for solutions to~\eqref{eq:model}, given our hypotheses.
\begin{lemma}[Main estimates]\label{lem:apriori}
Assume that Hypotheses~\ref{hyp:hemicont},~\ref{hyp:coerc},~\ref{hyp:coerc-g},~\ref{hyp:monotonicity}, \ref{hyp:monotonicity-g}, \ref{hyp:growth}, \ref{hyp:regularity} are valid. Let $T>0$ be fixed.
Let $\alpha\ge 2$, $\alpha-2\ge\beta\ge 0$ be as in Hypotheses \ref{hyp:coerc} and \ref{hyp:monotonicity}.
If $\alpha=\beta+2$, we assume additionally that
\[\|B\|_{L_2(U,H)}^2\le\lambda_3\frac{1}{\alpha}\delta_1 c_0^\alpha,\]
where $\lambda_3$ is as in Hypothesis~\ref{hyp:quant}, and $\delta_1>0$ is as in 
Hypothesis~\ref{hyp:coerc}.
Then for any initial datum $X_0=x$ such that $\exp(\|x\|^{\beta+2}_H)\in L^{1}(\Omega,\Fcal_0,\P;H)$, and every $0\le s\le t\le T$, and for any $n\in\N$, the solution $(X_t^{(n),x_n})_{t\in [0,T]}$ to~\eqref{eq:finitedim0new} satisfies
\begin{equation}\label{eq:apriori}
\begin{split}
&\E\left[\exp\left(\|X_t^{(n),x_n}\|^{\beta+2}_H+\lambda_0\delta_1\frac{\beta+2}{2}\int_s^{t} \|X_r^{(n),x_n}\|^\beta_H \|X_r^{(n),x_n}\|^\alpha_V\,\ud r\right)\right]\\
\le&\E\left[\exp\left(\|X_s^{(n),x_n}\|_H^{\beta+2}+(t-s)(c_1+c_2+c_3)\right)\right],
\end{split}
\end{equation}
for every $0\le s\le t\le T$ in the locally monotone case, and similarly with $\E$ replaced by $\tilde{\E}$ in the fully locally monotone case.
Here, $\lambda_0\in (0,1)$ and $c_i\ge 0$, $i=1,2,3$, are as in 
Hypothesis~\ref{hyp:quant}.

In particular, for $x_n:=Q_n x$, where $Q_n$ is defined as in \eqref{eq:Qn}, and for any $n\in\N$,
\begin{equation}\label{eq:apriori2}
\begin{split}
&\E\left[\exp\left(\|X_t^{(n),x_n}\|^{\beta+2}_H+\lambda_0\delta_1\frac{\beta+2}{2}\int_0^{t} \|X_r^{(n),x_n}\|^\beta_H \|X_r^{(n),x_n}\|^\alpha_V\,\ud r\right)\right]\\
\le&\E\left[\exp\left(\|x_n\|_H^{\beta+2}+t(c_1+c_2+c_3)\right)\right],
\end{split}
\end{equation}
for every $0\le t\le T$ in the locally monotone case, and similarly with $\E$ replaced by $\tilde{\E}$ in the fully locally monotone case.

Furthermore, we obtain,
\begin{equation}\label{eq:apriori3}
\begin{split}
&\E\left[\|X_t\|^{\beta+2}_H\right]+\lambda_0\delta_1 c_0^\alpha\frac{\beta+2}{2}\E\int_0^{t} \|X_r\|^{\alpha+\beta}_H\,\ud r\\
\le&\E\left[\|x\|_H^{\beta+2}\right]+t(c_1+c_2),
\end{split}
\end{equation}
for every $0\le t\le T$, and similarly with $\E$ replaced by $\tilde{\E}$ in the fully locally monotone case.

Also,
\begin{equation}\label{eq:apriori4}
\begin{split}
&\E\left[\|X_t\|^{2}_H\right]+\delta_1\E\int_0^{t} \|X_r\|^{\alpha}_V\,\ud r\\
\le&\E\left[\|x\|_H^{2}\right]+t\|B\|^2_{L_2(U,H)},
\end{split}
\end{equation}
for every $0\le t\le T$, and similarly with $\E$ replaced by $\tilde{\E}$ in the fully locally monotone case.
\end{lemma}
\begin{proof}
Let $(X^{(n),x_n}_t)_{t\geq 0}$ be the solution to \eqref{eq:finitedim0new} with $x_n:=Q_n x$.
We remark that $X^{(n),x_n}_t\in H_n\subset V$ for all $t\geq 0$.
For all $n\in \mathbb{N}$ we note that  $\|Q_n y\|_{H_n}\leq \|y\|_H$ for all $y\in H$, and $\|z\|_{H_n}=\|z\|_{H}$ for all $z\in H_n$.
Since we are interested in moments estimates, we introduce the following localization argument.
Let $T>0$ be a fixed 
time-horizon and take $R>\|x_n\|_{H_n}+1$. We then define the stopping time
\begin{equation}\label{eq:stop}
\tau^{(n),x_n}_{R}:=\inf\left\{s\geq 0: \|X^{(n),x_n}_s\|_{H}> R\right\}\wedge T.
\end{equation}
Similarly as in the proofs of~\cite[Theorem~5.1.3]{LR:15} and \cite[Theorem 2.6]{RSZ:24}, we have
\begin{equation}\label{eq:tau1}
\tau^{(n),x_n}_R\uparrow T\quad \mathbb{P}\textrm{-a.s. as $R\uparrow \infty$} 
\end{equation}
and 
\begin{equation}\label{eq:tau2}
\lim\limits_{R\to \infty}\sup_{n\in \mathbb{N}} \mathbb{P}(\tau^{(n),x_n}_R<T)=0.
\end{equation}
For the sake of readability, let us abbreviate $\tau_R:=\tau^{(n),x_n}_R$, and $X^{(n)}:=X^{(n),x_n}$.
By It\^o's 
formula~\cite{GK:82, M:82,GS:17}, or~\cite[Section~2.8]{S:05},
it follows that $\mathbb{P}$-a.s.
\begin{equation}\label{eq:Itoformula}
\begin{split}
\|X_t^{(n)}\|^{\beta+2}_H 
=&\|X_s^{(n)}\|_H^{\beta+2}+\frac{\beta+2}{2}\int_s^t 2 \|X_r^{(n)}\|^\beta_H \langle A(X_r^{(n)}), X_r^{(n)}\rangle\,\ud r\\
&+\frac{\beta+2}{2}\int_s^t \|X_r^{(n)}\|^\beta_H \|Q_n B \tilde{Q}_n\|^{2}_{L_2(U,H)}\,\ud r\\
&+\frac{\beta(\beta+2)}{2}\int_s^t \|X_r^{(n)}\|^{\beta-2}_H\|(Q_n B \tilde{Q}_n)^\ast X^{(n)}_{r}\|^2_{U}\,\ud r\\
&+(\beta+2)\int_s^t \|X_r^{(n)}\|^{\beta}_H\langle X_r^{(n)},Q_n B\,\ud W_r^{(n)}\rangle,
\end{split}
\end{equation}
for all $0\le s\le t\le T$.
And thus, for any $0\le s\le t\le T$, $\P$-a.s., by Hypothesis~\ref{hyp:coerc}, and Fact \ref{rem:embedding}, for $\lambda_0\in (0,1)$,
\begin{align*}
&\|X_t^{(n)}\|^{\beta+2}_H +\lambda_0\delta_1\frac{\beta+2}{2}\int_s^t \|X_r^{(n)}\|^\beta_H \|X_r^{(n)}\|^\alpha_V\,\ud r+(1-\lambda_0)\delta_1 c_0^\alpha\frac{\beta+2}{2}\int_s^t \|X_r^{(n)}\|^{\alpha+\beta}_H \,\ud r\\
\le&\|X_t^{(n)}\|^{\beta+2}_H +\delta_1\frac{\beta+2}{2}\int_s^t \|X_r^{(n)}\|^\beta_H \|X_r^{(n)}\|^\alpha_V\,\ud r\\
\le&\|X_s^{(n)}\|_H^{\beta+2}
+\frac{\beta+2}{2}\int_s^t \|X_r^{(n)}\|^\beta_H \|Q_n B \tilde{Q}_n\|^{2}_{L_2(U,H)}\,\ud r\\
&+\frac{\beta(\beta+2)}{2}\int_s^t \|X_r^{(n)}\|^{\beta-2}_H\|(Q_n B \tilde{Q}_n)^\ast X^{(n)}_{r}\|^2_{U}\,\ud r\\
&+(\beta+2)\int_s^t \|X_r^{(n)}\|^{\beta}_H\langle X_r^{(n)},Q_n B\,\ud W_r^{(n)}\rangle\\
=&\|X_s^{(n)}\|_H^{\beta+2}
+\frac{\beta+2}{2}\int_s^t \|X_r^{(n)}\|^\beta_H \|Q_n B \tilde{Q}_n\|^{2}_{L_2(U,H)}\,\ud r\\
&+\frac{\beta(\beta+2)}{2}\int_s^t \|X_r^{(n)}\|^{\beta-2}_H\|(Q_n B \tilde{Q}_n)^\ast X^{(n)}_{r}\|^2_{U}\,\ud r\\
&+(\beta+2)\int_s^t \|X_r^{(n)}\|^{\beta}_H\langle X_r^{(n)},Q_n B\,\ud W_r^{(n)}\rangle-\frac{(\beta+2)^2}{2}\int_s^t \|X_r^{(n)}\|^{2\beta+2}_H\|Q_n B\tilde{Q}_n\|_{L_2(U,H)}^2\,\ud r\\
&+\frac{(\beta+2)^2}{2}\int_s^t \|X_r^{(n)}\|^{2\beta+2}_H\|Q_n B\tilde{Q}_n\|_{L_2(U,H)}^2\,\ud r\\
\le&\|X_s^{(n)}\|_H^{\beta+2}
+\frac{\beta+2}{2}\int_s^t \|X_r^{(n)}\|^\beta_H \|B\|^{2}_{L_2(U,H)}\,\ud r\\
&+\frac{\beta(\beta+2)}{2}\int_s^t \|X_r^{(n)}\|^{\beta}_H\|B \|^2_{L(U,H)}\,\ud r\\
&+(\beta+2)\int_s^t \|X_r^{(n)}\|^{\beta}_H\langle X_r^{(n)},Q_n B\,\ud W_r^{(n)}\rangle-\frac{(\beta+2)^2}{2}\int_s^t \|X_r^{(n)}\|^{2\beta+2}_H\|Q_n B\tilde{Q}_n\|_{L_2(U,H)}^2\,\ud r\\
&+\frac{(\beta+2)^2}{2}\int_s^t \|X_r^{(n)}\|^{2\beta+2}_H\|B\|_{L_2(U,H)}^2\,\ud r.
\end{align*}
Recall Young's inequality for products, that is, for any $p,q\in (1,\infty)$ satisfying $p^{-1}+q^{-1}=1$ it follows that
\begin{equation}\label{eq:Young}|xy|\le\eps|x|^p+\frac{(p\eps)^{1-q}}{q}|y|^q
\quad \text{ for any }\quad x,y\in \mathbb{R}\quad \textrm{ and }\quad \eps>0.
\end{equation}
For any $y\in H$, and any $\lambda_i\in (0,1)$, $i=1,2,3$, and $\lambda_0\in (0,1)$ with $\sum_{i=0}^3\lambda_i=1$, and $c_i\ge 0$, $i=1,2,3$ as follows, compare also with Hypothesis \ref{hyp:quant},
\begin{equation}
\begin{split}
&c_1:=\frac{\alpha(\beta+2)}{2(\alpha+\beta)\left(\lambda_1\delta_1 c_0^\alpha\frac{\alpha+\beta}{\beta}\right)^{\frac{\beta}{\alpha}}}\|B\|_{L_2(U,H)}^{\frac{2(\alpha+\beta)}{\alpha}},\\
&c_2:=\beta^{\frac{\alpha+\beta}{\alpha}}\frac{\alpha(\beta+2)}{2(\alpha+\beta)\left(\lambda_2\delta_1 c_0^\alpha\frac{\alpha+\beta}{\beta}\right)^{\frac{\beta}{\alpha}}}\|B\|_{L(U,H)}^{\frac{2(\alpha+\beta)}{\alpha}},\\
&c_3:=\frac{(\beta+2)(\alpha-\beta-2)}{2(\alpha+\beta)\left(\lambda_3\delta_1 c_0^\alpha\frac{\alpha+\beta}{2\beta+2}\right)^{\frac{2\beta+2}{\alpha-\beta-2}}}(\beta+2)^{\frac{\alpha+\beta}{\alpha-\beta-2}}\|B\|_{L_2(U,H)}^{\frac{2(\alpha+\beta)}{\alpha-\beta-2}},
\end{split}
\end{equation}
we get the following inequalities by~\eqref{eq:Young}, if $\beta>0$,
\begin{equation}
\begin{split}
&\frac{\beta+2}{2}\|y\|_H^\beta \|B\|_{L_2(U,H)}^2\le\lambda_1\frac{\beta+2}{2}\delta_1 c_0^\alpha\|y\|_H^{\alpha+\beta}+c_1,\\
&\frac{(\beta+2)\beta}{2}\|y\|^\beta_H\|B\|^2_{L(U,H)}\le\lambda_2\frac{\beta+2}{2}\delta_1 c_0^\alpha\|y\|_H^{\alpha+\beta}+c_2,
\end{split}
\end{equation}
and otherwise, if $\beta=0$, then $\lambda_1=\lambda_2=0$ and,
\begin{equation}
\begin{split}
&c_1=\|B\|_{L_2(U,H)}^2\\
&c_2=0,
\end{split}
\end{equation}
furthermore, if $\alpha+\beta>2\beta+2$,
\begin{equation}
\begin{split}
&\frac{(\beta+2)^2}{2}\|y\|_H^{2\beta+2} \|B\|_{L_2(U,H)}^2\le\lambda_3\frac{\beta+2}{2}\delta_1 c_0^\alpha\|y\|_H^{\alpha+\beta}+c_3,
\end{split}
\end{equation}
and otherwise, if $\alpha=\beta+2$, then $c_3=0$, and by assumption,
\begin{equation}
\begin{split}
&\|B\|_{L_2(U,H)}^2\le\lambda_3\frac{1}{\alpha}\delta_1 c_0^\alpha,
\end{split}
\end{equation}
which, in this case, is equivalent to
\begin{equation}
\begin{split}
&\frac{(\beta+2)^2}{2}\|B\|_{L_2(U,H)}^2\le\lambda_3\frac{\beta+2}{2}\delta_1 c_0^\alpha.
\end{split}
\end{equation}
Combining these inequalities,
we have that $\mathbb{P}$-a.s.
\begin{align*}
&\|X_t^{(n)}\|^{\beta+2}_H +\lambda_0\delta_1\frac{\beta+2}{2}\int_s^t \|X_r^{(n)}\|^\beta_H \|X_r^{(n)}\|^\alpha_V\,\ud r\\
\le&\|X_s^{(n)}\|_H^{\beta+2}+(t-s)(c_1+c_2+c_3)\\
&+(\beta+2)\int_s^t \|X_r^{(n)}\|^{\beta}_H\langle X_r^{(n)},Q_n B\,\ud W_r^{(n)}\rangle-\frac{(\beta+2)^2}{2}\int_s^t \|X_r^{(n)}\|^{2\beta+2}_H\|Q_n B\tilde{Q}_n\|_{L_2(U,H)}^2\,\ud r.
\end{align*}
Stopping and taking the exponential yields $\mathbb{P}$-a.s.,
\begin{align*}
&\exp\left(\|X_{t\wedge\tau_R}^{(n)}\|^{\beta+2}_H+\lambda_0\delta_1\frac{\beta+2}{2}\int_s^{t\wedge\tau_R} \|X_r^{(n)}\|^\beta_H \|X_r^{(n)}\|^\alpha_V\,\ud r\right)\\
\le&\exp\left(\|X_s^{(n)}\|_H^{\beta+2}+((t\wedge\tau_R)-s)(c_1+c_2+c_3)\right)\\
&\times\exp\left((\beta+2)\int_s^{t\wedge\tau_R} \|X_r^{(n)}\|^{\beta}_H\langle X_r^{(n)},Q_n B\,\ud W_r^{(n)}\rangle\right.\\
&\left.-\frac{(\beta+2)^2}{2}\int_s^{t\wedge\tau_R} \|X_r^{(n)}\|^{2\beta+2}_H\|Q_n B\tilde{Q}_n\|_{L_2(U,H)}^2\,\ud r\right)\\
\le&\exp\left(\|X_s^{(n)}\|_H^{\beta+2}+(t-s)(c_1+c_2+c_3)\right)\\
&\times\exp\left((\beta+2)\int_s^{t\wedge\tau_R} \|X_r^{(n)}\|^{\beta}_H\langle X_r^{(n)},Q_n B\,\ud W_r^{(n)}\rangle\right.\\
&\left.-\frac{(\beta+2)^2}{2}\int_s^{t\wedge\tau_R} \|X_r^{(n)}\|^{2\beta+2}_H\|Q_n B\tilde{Q}_n\|_{L_2(U,H)}^2\,\ud r\right).
\end{align*}
The stopped exponential is an exponential martingale which has expectation less or equal to one, see e.g. \cite{S:05}. As it is independent from $X_s^{(n)}$, we obtain,
\begin{align*}
&\E\left[\exp\left(\|X_{t\wedge\tau_R}^{(n)}\|^{\beta+2}_H+\lambda_0\delta_1\frac{\beta+2}{2}\int_s^{t\wedge\tau_R} \|X_r^{(n)}\|^\beta_H \|X_r^{(n)}\|^\alpha_V\,\ud r\right)\right]\\
\le&\E\left[\exp\left(\|X_s^{(n)}\|_H^{\beta+2}+(t-s)(c_1+c_2+c_3)\right)\right].
\end{align*}
By~\eqref{eq:tau1},~\eqref{eq:tau2},~and~Fatou's Lemma, we get that
\begin{align*}
&\E\left[\exp\left(\|X_t^{(n)}\|^{\beta+2}_H+\lambda_0\delta_1\frac{\beta+2}{2}\int_s^{t} \|X_r^{(n)}\|^\beta_H \|X_r^{(n)}\|^\alpha_V\,\ud r\right)\right]\\
\le&\E\left[\exp\left(\|X_s^{(n)}\|_H^{\beta+2}+(t-s)(c_1+c_2+c_3)\right)\right].
\end{align*}
Estimate \eqref{eq:apriori} is proved. Estimate \eqref{eq:apriori2} follows by setting $s=0$.
To prove \eqref{eq:apriori3}, note that by Fact \ref{rem:embedding},
\begin{align*}
&\E\left[\|X_t^{(n)}\|^{\beta+2}_H\right]+\lambda_0\delta_1 c_0^\alpha\frac{\beta+2}{2}\E\int_0^{t} \|X_r^{(n)}\|^{\alpha+\beta}_H \,\ud r\\
\le&\E\left[\|x_n\|_H^{\beta+2}\right]+t(c_1+c_2).
\end{align*}
Note that stopping and taking the expectation removes the stochastic integral. We have that
$\|Q_n y\|_H^2\le\|y\|_H^2$ for any $y\in H$. Therefore $\{X^{(n)}\}$ is uniformly bounded in $L^{\alpha+\beta}([0,T]\times\Omega;H)$ and in $L^\infty([0,T];L^{\beta+2}(\Omega;H))$, and therefore a subsequence converges weakly in $L^{\alpha+\beta}([0,T]\times\Omega;H)$, and weakly$^\ast$ in $L^\infty([0,T];L^{\beta+2}(\Omega;H))$, respectively. By \eqref{eq:Qconvergence} and the density of $V\subset H$, we obtain that the limit in both cases is $X$, which proves \eqref{eq:apriori3} for $\ud t$-almost every $0\le t\le T$, by passing to the limit for the subsequence, and similarly with $\E$ replaced by $\tilde{\E}$ in the fully locally monotone case, where we need to use Fatou's lemma.

To prove \eqref{eq:apriori4}, note that we can apply It\^o's formula above for $\beta=0$ and obtain the following. Again, stopping and taking the expectation removes the stochastic integral.
\begin{align*}
&\E\left[\|X_t^{(n)}\|^{2}_H\right]+\delta_1 \E\int_0^{t} \|X_r^{(n)}\|^{\alpha}_V \,\ud r\\
\le&\E\left[\|x_n\|_H^{2}\right]+t\|B\|^2_{L_2(U,H)}\\
\le&\E\left[\|x\|_H^{2}\right]+t\|B\|^2_{L_2(U,H)}.
\end{align*}
We see that $\{X^{(n)}\}$ is uniformly bounded in $L^{\alpha}([0,T]\times\Omega;V)$, and thus a subsequence converges weakly in $L^{\alpha}([0,T]\times\Omega;V)$. By \eqref{eq:Qconvergence} and the density of $V\subset H$, we obtain that the limit is $X$.
Passing to the limit $k\to\infty$ for the subsequence yields
\begin{align*}
&\E\left[\|X_t\|^{2}_H\right]+\delta_1 \E\int_0^{t} \|X_r\|^{\alpha}_V \,\ud r\\
\le&\E\left[\|x\|_H^{2}\right]+t\|B\|^2_{L_2(U,H)}.
\end{align*}
for $\ud t$-almost every $0\le t\le T$ in the locally monotone case, and similarly with $\E$ replaced by $\tilde{\E}$ in the fully locally monotone case, where we need to use Fatou's lemma.

By continuity of the solutions, see \cite{LR:15,RSZ:24}, the right-hand side depends continuously on $t$, we get all estimates for every $0\le t\le T$ by a simple approximation argument. The proof is complete.
\end{proof}

\subsection{Properties of the semigroup}

We collect some basic properties of the semigroup.
Let us verify the $e$-property of the semigroup $(P_t)_{t\ge 0}$ first.
\begin{proposition}[$e$-property]\label{prop:e}
Assume that Hypotheses~\ref{hyp:hemicont},~\ref{hyp:coerc},~\ref{hyp:coerc-g},~\ref{hyp:monotonicity},~\ref{hyp:monotonicity-g},~\ref{hyp:growth}, \ref{hyp:angle}, \ref{hyp:regularity}, and \ref{hyp:quant} are valid.
Let $\eps>0$, $F\in\Lip_b(H)$, $x\in H$. Then there exists $\delta>0$ such that for any $y\in H$ with $\|x-y\|<\delta$, and any $t\ge 0$, we have that
\[|P_t F(x)-P_t F(y)|<\eps.\]
\end{proposition}
\begin{proof}
Note that if $C_2=0$, the $e$-property follows as in~\cite[Equation~(5.4)]{GT:14}.
Let $\eps>0$, $x,y\in H$, $\|x-y\|<\delta$, $F\in\Lip_b(H)$, $t\ge 0$. Let $(X^x_t)_{t\ge 0}$, $(X^y_t)_{t\ge 0}$ denote the solutions to~\eqref{eq:model} with $X_0^x=x$, and $X^y_0=y$, respectively.
For $n\in\N$, let $(X^{(n),x_n}_t)_{t\geq 0}$ be the solution to \eqref{eq:finitedim0new} with $x_n:=Q_n x$, and let $(X^{(n),y_n}_t)_{t\geq 0}$ be the solution to \eqref{eq:finitedim0new} with $y_n:=Q_n y$.

Note that due to the additivity of the noise, the synchronous coupling (a.k.a. parallel coupling) yields that the map $t\mapsto (X^{(n),x_n}_t-X^{(n),y_n}_t)$ is absolutely continuous in $H_n$.
The chain rule \cite[Proposition III.1.2]{Show} with the help of Hypothesis \ref{hyp:monotonicity} yields that $\P$-a.s.,
\begin{align*}
\frac{\ud}{\ud t}\|X^{(n),x_n}_t-X^{(n),y_n}_t\|_H^2=&2\langle A(X^{(n),x_n}_t)-A(X^{(n),y_n}_t),X^{(n),x_n}_t-X^{(n),y_n}_t\rangle \\
\le&\left(-\delta_2 +\eta(X^{(n),x_n}_t) +\rho(X^{(n),y_n}_t)\right)\|X^{(n),x_n}_t-X^{(n),y_n}_t\|^2_H,
\end{align*}
for all $t\in [0,T]$.
Similarly, we get (by interchanging the roles of $x_n$ and $y_n$) that
\begin{align*}
\frac{\ud}{\ud t}\|X^{(n),y_n}_t-X^{(n),x_n}_t\|_H^2=&2\langle A(X^{(n),y_n}_t)-A(X^{(n),x_n}_t),X^{(n),y_n}_t-X^{(n),x_n}_t\rangle \\
\le&\left(-\delta_2 +\eta(X^{(n),y_n}_t) +\rho(X^{(n),x_n}_t)\right)\|X^{(n),y_n}_t-X^{(n),x_n}_t\|^2_H,
\end{align*}
for all $t\in [0,T]$.
The preceding two differential inequalities with the help of Hypothesis~\ref{hyp:monotonicity} yields 
\begin{align*}
\frac{\ud}{\ud t}\|X^{(n),x_n}_t-X^{(n),y_n}_t\|_H^2
\le&\left(-\delta_2 +\frac{1}{2}C_2\|X^{(n),y_n}_t\|^\alpha_V \|X^{(n),y_n}_t\|^\beta_H+\frac{1}{2}C_2\|X^{(n),x_n}_t\|^\alpha_V\|X^{(n),x_n}_t\|^\beta_H\right)\\
&\quad\times\|X^{(n),x_n}_t-X^{(n),y_n}_t\|^2_H,
\end{align*}
for all $t\in [0,T]$.

An application of the differential form of Gronwall's lemma yields
\begin{align*}
&\|X^{(n),x_n}_t-X^{(n),y_n}_t\|_H^2\\
\le&\|x_n-y_n\|^2_H\exp\left(-\delta_2 t +\frac{1}{2}C_2\int_0^t\left[\|X^{(n),x_n}_t\|^\alpha_V \|X^{(n),x_n}_t\|^\beta_H+\|X^{(n),y_n}_r\|^\alpha_V\|X^{(n),y_n}_r\|^\beta_H\right]\,\ud r\right),
\end{align*}
for all $t\in [0,T]$ and all $n\in\N$.

Now, by the Cauchy-Schwarz inequality, Lemma~\ref{lem:apriori} and Hypothesis~\ref{hyp:quant},
\begin{align*}
&\E\left[\|X^{(n),x_n}_t-X^{(n),y_n}_t\|_H^2\right]\\
\le&\|x_n-y_n\|^2_H\exp\left(-\delta_2 t\right)\\
&\times\left(\E\left[\exp\left(C_2\int_0^t \|X^{(n),x_n}_t\|^\alpha_V \|X^{(n),x_n}_t\|^\beta_H\,\ud r\right)\right]\right)^{\frac{1}{2}}\left(\E\left[\exp\left(C_2\int_0^t\|X^{(n),y_n}_r\|^\alpha_V\|X^{(n),y_n}_r\|^\beta_H\,\ud r\right)\right]\right)^{\frac{1}{2}}\\
\le&\|x_n-y_n\|^2_H\exp\left(-\delta_2 t +\frac{C_2}{\lambda_0\delta_1(\beta+2)}\left(\|x_n\|_H^{\beta+2}+\|y_n\|_H^{\beta+2}+2t(c_1+c_2+c_3)\right)\right)\\
\le&\|x_n-y_n\|^2_H\exp\left(\frac{C_2}{\lambda_0\delta_1(\beta+2)}\left(\|x_n\|_H^{\beta+2}+\|y_n\|_H^{\beta+2}\right)\right)\\
\le&\|x_n-y_n\|^2_H\exp\left(\frac{C_2}{\lambda_0\delta_1(\beta+2)}\left(\|x_n\|_H^{\beta+2}+2^{\beta+1}\|y_n-x_n\|_H^{\beta+2}+2^{\beta+1}\|x_n\|_H^{\beta+2}\right)\right)\\
<&\delta^2\exp\left(\frac{C_2}{\lambda_0\delta_1(\beta+2)}\left(\left(1+2^{\beta+1}\right)\|x_n\|_H^{\beta+2}+2^{\beta+1}\delta^{\beta+2}\right)\right).\end{align*}
In the locally monotone case, by~\eqref{eq:Qconvergence}, there exists a subsequence $\{n_k\}$ such that $Q_{n_k} X^{(n_k),x_{n_k}}\rightharpoonup^\ast X^{x}$ as $k\to\infty$, and we may extract a further subsequence $\{n_{k_l}\}$ such that $Q_{n_{k_l}} X^{(n_{k_l}),y_{n_{k_l}}}\rightharpoonup^\ast X^{y}$ weakly$^\ast$ in $L^\infty([0,T];L^2(\Omega;H))$.
In the fully locally monotone case, the convergence is $\tilde{\P}$-a.s. strongly in $L^{\hat{\alpha}}([0,T],H)$. Recall that by Hypothesis \ref{hyp:coerc-g}, $\hat{\alpha}\ge 2$. Note that $\|Q_n y\|_H^2\le\|y\|_H^2$ for any $y\in H$.
We may pass to the limit $l\to \infty$, and obtain,
\begin{align*}
&\E\left[\|X^{x}_t-X^{y}_t\|_H^2\right]\\
<&\delta^2\exp\left(\frac{C_2}{\lambda_0\delta_1(\beta+2)}\left(\left(1+2^{\beta+1}\right)\|x\|_H^{\beta+2}+2^{\beta+1}\delta^{\beta+2}\right)\right).\end{align*}
and similarly with $\E$ replaced by $\tilde{\E}$ in the fully locally monotone case, where we need to use Fatou's lemma.

Now,
\begin{align*}
&|P_t F(x)-P_t F(y)|^2\\
=&|\E\left[F(X^x_t)\right]-\E\left[F(X^y_t)\right]|^2\\
\le&\E\left[|F(X^x_t)-F(X^y_t)|^2\right]\\
\le&\Lip(F)^2\E[\left\|X^x_t-X^y_t\|^2_H\right]\\
<&\Lip(F)^2\delta^2\exp\left(\frac{C_2}{\lambda_0\delta_1(\beta+2)}\left(\left(1+2^{\beta+1}\right)\|x\|_H^{\beta+2}+2^{\beta+1}\delta^{\beta+2}\right)\right)\\
<&\varepsilon^2,
\end{align*}
when $\delta>0$ is small enough.
\end{proof}

\begin{lemma}
Assume that Hypotheses~\ref{hyp:hemicont},~\ref{hyp:coerc},~\ref{hyp:coerc-g},~\ref{hyp:monotonicity},~\ref{hyp:monotonicity-g},~\ref{hyp:growth}, \ref{hyp:angle}, \ref{hyp:regularity}, and \ref{hyp:quant} are valid.
Then the semigroup $(P_t)_{t\ge 0}$ associated to~\eqref{eq:model} is a stochastically continuous Markovian Feller semigroup.
\end{lemma}
\begin{proof}
Let us prove the Feller property.
Let $F\in C_b(H)$, $t\ge 0$. Hence, we get that,
\begin{align*}
&\|P_t F(\cdot)\|_\infty=\sup_{x\in H}|\E[F(X_t^x)]|\le\sup_{y\in H}\|F(y)\|_H<\infty.
\end{align*}
Let $x_n,x\in H$, $n\in\N$ such that $\|x_n-x\|_H\to 0$ as $n\to\infty$. Let $t\ge 0$ and $F\in C_b(H)$. Let $F_m\in\Lip_b(H)$, $m\in\N$ with $\|F_m(\cdot)-F(\cdot)\|_\infty\to 0$ as $m\to\infty$. Then, for any $\eps>0$, there exists $m\in \N$ such that
\[\|F_m(\cdot)-F(\cdot)\|_\infty<\frac{\eps}{3}.\]
By the $e$-property proved in 
Proposition~\ref{prop:e}, for any $\delta>0$ there exists $n_0\in\N$ such that $\|x_n-x\|_H<\delta$ for all $n\ge n_0$, and thus
\[|P_t F_m(x_n)-P_t F_m(x)|<\frac{\eps}{3}\]
for $n\ge n_0(\delta)$ and all $t\ge 0$, whenever $\delta=\delta(\eps,m)>0$ is small enough.
As a consequence,
\begin{align*}
&|P_t F(x_n)-P_t F(x)|\\
\le&|P_t F(x_n)-P_t F_m(x_n)|+|P_t F_m(x_n)-P_t F_m(x)|+|P_t F_m(x)-P_t F(x)|\\
=&|\E[F(X_t^{x_n})-F_m(X_t^{x_n})]|+|P_t F_m(x_n)-P_t F_m(x)|+|\E[F_m(X_t^{x})-F(X_t^{x})]|\\
\le&2\sup_{y\in H}\|F_m(y)-F(y)\|_H+\frac{\eps}{3}<\eps.
\end{align*}
Let us prove the stochastic continuity. Let $x\in H$.
By~\cite[Proposition 2.1.1]{DPZ:96}, it is sufficient to prove that
\[|P_t F(x)-F(x)|\to 0\]
as $t\searrow 0$ for any $F\in\Lip_b(H)$. Clearly,
\[|P_t F(x)-F(x)|^2\le\Lip(F)^2\E[\|X_t^x-x\|^2_H].\]
By an adaptation of 
Lemma~\ref{lem:apriori},
\begin{align*}
\E\left[\|X_t^x-x\|^{2}_H\right]\le\|B\|^2_{L_2(U,H)}t\longrightarrow 0,
\end{align*}
as $t\searrow 0$.

The Markov property follows by the same arguments as in~\cite[Proposition~4.3.5]{LR:15}, see also~\cite[Section~6.4]{GT:14}.
\end{proof}

\subsection{Uniqueness of invariant measures}

Our aim is to verify the 
conditions~\eqref{eq:KPS1} 
and~\eqref{eq:KPS2} of Lemma~\ref{lem:KPS}, compare 
with~\cite{GT:14,ESR:12}.

For a bounded subset $J\subset H$, set $\|J\|_H:=\sup_{x\in J}\|x\|_H$.
Set
\[P_t(x,B):=P^\ast_t\delta_x(B),\quad t\ge 0,\quad x\in H,\quad B\in\Bcal(H).\] 
Consider the measurable Lyapunov function
\begin{equation}\label{eq:theta}\Theta(x):=c^\alpha_0\delta_1\|x\|_H^\alpha,\end{equation}
which has bounded sublevel sets.
Consider the locally monotone PDE
\[\frac{\ud}{\ud t} u^x_t=A(u^x_t),\quad u^x_0=x\in  H.\]
The existence and uniqueness of solutions in $C([0,T];H)$ follows under 
Hypotheses~\ref{hyp:hemicont},~\ref{hyp:coerc},~\ref{hyp:monotonicity},~\ref{hyp:growth} by~\cite[Theorem~5.1.3]{LR:15}.

\begin{lemma}\label{lem:decay}
Assume that Hypotheses~\ref{hyp:hemicont},~\ref{hyp:coerc},~\ref{hyp:monotonicity},~\ref{hyp:growth},~and~\ref{hyp:regularity} hold true. For every $R>0$,
\[\lim_{t\to\infty}\sup_{x\in B_R}\|u^x_t\|_H= 0,\]
where $B_R:=\{y\in H:\Theta(y)\le R\}$, where $\Theta$ is as in \eqref{eq:theta}.
\end{lemma}
\begin{proof}
Let $R>0$ and $x\in B_R$.
By Hypothesis~\ref{hyp:coerc} and 
Fact~\ref{rem:embedding} we have
\[
\frac{\ud}{\ud t}\|u_t^x\|_H^2=2\langle A(u^x_t),u^x_t\rangle \le-\delta_1\|u_t^x\|^\alpha_V\le -c^\alpha_0\delta_1\left(\|u_t^x\|_H^2\right)^{\alpha/2}.
\]
Note that $y(t):=\|u_t^x\|^2_H$ is a subsolution to the ordinary differential equation
\[y'=-c^\alpha_0\delta_1 y^{\alpha/2},\quad\text{for a.e.}\;t\in [0,T],\quad y(0)=\|x\|^2_H,\]
with solution for $\alpha>2$,
\[y(t)=\begin{cases}\left(\|x\|^{2-\alpha}_H+\frac{\alpha-2}{2}c^\alpha_0\delta_1 t\right)^{-2/(\alpha-2)},&\text{if}\quad x\not=0,\\
0,&\text{if}\quad x= 0,\end{cases}\]
and solution for $\alpha=2$,
\[y(t)=\|x\|_H^2\exp(-c_0^2\delta_1 t).\]
Hence, by a standard comparison principle, see e.g. \cite{W:13}, we have for $\alpha>2$ that
\begin{equation}
\begin{split}
\|u_t^x\|_H&\le \left(\|x\|^{2-\alpha}_H+\frac{\alpha-2}{2}c^\alpha_0\delta_1 t\right)^{-1/(\alpha-2)}\\
&\le \left(\left(
\frac{R}{c_0^\alpha \delta_1}\right)^{(2-\alpha)/\alpha}+\frac{\alpha-2}{2}c^\alpha_0\delta_1 t\right)^{-1/(\alpha-2)},
\end{split}
\end{equation}
and for $\alpha=2$ that
\[\|u_t^x\|_H\le \left(\|x\|^{2}_H\exp(-c_0^2\delta_1 t)\right)^{1/2}\le \frac{\sqrt{R}}{c_0\sqrt{\delta_1}}\exp\left(-\frac{1}{2}c_0^2\delta_1 t\right),\]
which proves the claim.
\end{proof}

We shall need an a priori bound for the deterministic solution.
\begin{lemma}\label{lem:u-bound}
Assume that Hypotheses~\ref{hyp:hemicont},~\ref{hyp:coerc},~\ref{hyp:coerc-g},~\ref{hyp:monotonicity},~\ref{hyp:monotonicity-g}, and \ref{hyp:growth} hold true. For any $T>0$ and any $x\in H$, we have that
\begin{equation}\label{eq:u-bound}
\begin{split}
&\|u_t^x\|^{\beta+2}_H+\delta_1 \frac{\beta+2}{2}\int_0^{t} \|u_r^x\|^{\beta}_H\|u_r^x\|_V^{\alpha}\,\ud r\\
\le&\|x\|_H^{\beta+2},
\end{split}
\end{equation}
where $\alpha$ and $\delta_1$ are as in Hypothesis \ref{hyp:coerc} and $\beta$ is as in Hypothesis \ref{hyp:monotonicity}.
\end{lemma}
\begin{proof}
Note that $t\mapsto u_t^x$ is absolutely continuous. The chain rule \cite[Proposition III.1.2]{Show} gives
\begin{align*}
\|u_t^x\|^{\beta+2}_H 
=\|x\|_H^{\beta+2}+\frac{\beta+2}{2}\int_0^t 2 \|u_r^x\|^\beta_H \langle A(u_r^x), u_r^x\rangle\,\ud r.\end{align*}
Hypothesis \ref{hyp:coerc} yields that
\begin{align*}
\|u_t^x\|^{\beta+2}_H  +\delta_1\frac{\beta+2}{2}\int_0^t 2 \|u_r^x\|^\beta_H \|u_r^x\|^\alpha_V\,\ud r
\le\|x\|_H^{\beta+2}.\end{align*}
\end{proof}

The next result shows uniform stochastic stability of the solutions to the SPDE and the deterministic PDE with positive probability.

\begin{lemma}\label{lem:coupling}
Assume that Hypotheses~\ref{hyp:hemicont},~\ref{hyp:coerc},~\ref{hyp:coerc-g},~\ref{hyp:monotonicity},~\ref{hyp:monotonicity-g},~\ref{hyp:growth},~\ref{hyp:angle},~\ref{hyp:regularity},~\ref{hyp:reg2}, and \ref{hyp:quant} hold true. Assume that in Hypothesis~\ref{hyp:monotonicity} $\rho\equiv 0$ or $\eta\equiv 0$ (locally monotone case). For any $T>0$, for any $\eps>0$, and any $K\subset H$ bounded, we have
\[\P\left(\|X_T^x-u_T^x\|_H^2\le\eps\right)>0,\]
uniformly for $x\in K$.
\end{lemma}
\begin{proof}
W.l.o.g. $\eta\equiv 0$. Let $T>0$, $\eps>0$, $x\in K$. Note that $Y_t^x:=X_t^x-B\,W_t$, $t\in [0,T]$, is absolutely continuous and solves the non-autonomous random PDE
\begin{equation}\label{eq:Y}
\frac{\ud}{\ud t}Y^x_t=A(Y^x_t+B\,W_t),\quad Y_0=x.
\end{equation}
By Hypothesis~\ref{hyp:reg2}, for any $\delta>0$, there exists an event $\Omega_\delta\in\Fcal$ with positive probability such that
\[\sup_{t\in [0,T]}\|B\,W_t(\omega)\|_V\le\frac{\delta}{2}\]
for all $\omega\in\Omega_\delta$. Let us assume that $0<\delta<\sqrt{\eps}$. Hence, for $\omega\in\Omega_\delta$, we have by Hypothesis~\ref{hyp:angle} and the chain rule that
\begin{equation}\label{eq:Yapriori}
\begin{split}
\|Y_t^x(\omega)\|_H^2=&
\|x\|^2_H+2\int_0^t\langle A(Y_s^x(\omega)+B\,W_s(\omega)),Y_s^x(\omega)\rangle\,\ud s
\\
=&\|x\|^2_H+2\int_0^t\langle A(X_s^x(\omega)),X_s^x(\omega)-B\,W_s(\omega)\rangle\,\ud s\\
\le&\|x\|^2_H-\delta_4\int_0^t\|A(X_s^x(\omega))\|_{V^\ast}\,\ud s-2\int_0^t \langle A(X_s^x(\omega)),B\,W_s(\omega)\rangle\,\ud s+(C_4\vee 0) t\\
\le&\|x\|^2_H-\left(\delta_4-2\sup_{s\in [0,t]}\|B\,W_s(\omega)\|_V\right)\int_0^t\|A(X^x_s(\omega))\|_{V^\ast}\,\ud s+(C_4\vee 0) t\\
\le&\|x\|_H^2-(\delta_4-\delta)\int_0^t\|A(X^x_s(\omega))\|_{V^\ast}\,\ud s+(C_4\vee 0)t,\quad t\in [0,T].
\end{split}\end{equation}
We proceed similarly for $u^x_t$. We get that there exists a constant $\widetilde{C}(T)>0$, depending on $T$, $\delta_4$ and $C_4$, such that
\begin{equation}\label{eq:Vastbound}
\int_0^T\left(\|A(u_s^x)\|_{V^\ast}+\|A(X^x_s(\omega))\|_{V^\ast}\right)\,\ud s\le \widetilde{C}(T)(1+\|x\|^2_H),\end{equation}
for $0<\delta<\frac{\delta_4}{2}$.
By the chain rule, Hypothesis \ref{hyp:monotonicity}, and a similar argument as in the proof of Proposition~\ref{prop:e},
\begin{align*}
&\|Y^x_t(\omega)-u_t^x\|^2_H\\
=&2\int_0^t\langle A(X^x_s(\omega))-A(u_s^x),X^x_s(\omega)-B\,W_s(\omega)-u_s^x\rangle\,\ud s\\
\le&\int_0^t\left(2\|Y^x_s(\omega)-u_s^x\|_H^2+2\|B\,W_s(\omega)\|_H^2\right)\left(-\delta_2+C_2\|u_s^x\|_V^\alpha\|u_s^x\|_H^\beta\right)\,\ud s\\
&+\|A(X^x(\omega))-A(u^x)\|_{L^1([0,T];V^\ast)}\|B\,W_t(\omega)\|_{L^\infty([0,T];V)}\\
\le&\int_0^t\left(2\|Y^x_s(\omega)-u_s^x\|_H^2+\frac{\delta^2}{2}\right)C_2\|u_s^x\|_V^\alpha\|u_s^x\|_H^\beta\,\ud s\\
&+\frac{\delta}{2}\widetilde{C}(T)\left(1+\sup_{y\in K}\|y\|^2_H\right).
\end{align*}
Let $\E_\delta:=\E_{\P(\cdot\,\vert\,\Omega_\delta)}$ be the expected value with respect to the conditional probability $\P(\cdot\,\vert\,\Omega_\delta)$. Set $\eta:=\frac{1}{2}\widetilde{C}(T)\left(1+\sup_{y\in K}\|y\|^2_H\right)$.
Now, by the stochastic Gronwall inequality \cite[Theorem 1.2]{G:24}, and by the Markov inequality, and by Lemma~\ref{lem:u-bound},
\begin{align*}
&\P\left(\left\{\sup_{0\le t\le T}\|Y^x_t-u_t^x\|^2_H >\frac{\eps}{4}\right\}\,\Bigg\vert\,\Omega_\delta\right)\\
\le&\frac{4}{\eps}e^{R}\E\left[\frac{\delta^2}{2}C_2\int_0^T\|u_s^x\|_V^\alpha\|u_s^x\|_H^\beta\,\ud s+
\delta\eta\right]+\P\left(\left\{C_2\int_0^T \|u_r^x\|_V^\alpha\|u_r^x\|_H^\beta\,\ud t> R\right\}\,\bigg\vert\,\Omega_\delta\right)\\
\le&\frac{4}{\eps}e^{R}\left(\frac{\delta^2 C_2}{\delta_1(\beta+2)}\sup_{y\in  K}\|y\|_H^{\beta+2}+\delta\eta\right)+\frac{1}{R}\E_\delta\left[\mathbbm{1}_{\Omega_\delta}C_2\int_0^T \|u_r^x\|_V^\alpha\|u_r^x\|_H^\beta\,\ud t\right]\\
\le&\frac{4}{\eps}e^{R}\left(\frac{\delta^2 C_2}{\delta_1(\beta+2)}\sup_{y\in K}\|y\|_H^{\beta+2}+\frac{\delta}{2}\widetilde{C}(T)\left(1+\sup_{y\in K}\|y\|^2_H\right)\right)+\frac{1}{R}\frac{2C_2}{\delta_1(\beta+2)}\sup_{y\in K}\|y\|_H^{\beta+2}.
\end{align*}
First, choose $R>0$ such second term is smaller than $\frac{1}{4}$. Then choose $0<\delta<\sqrt{\eps}\wedge\frac{\delta_4}{2}$ such that the first term is smaller than $\frac{1}{4}$.
We get that
\begin{align*}
&\P\left(\left\{\sup_{0\le t\le T}\|Y^x_t-u_t^x\|^2_H >\frac{\eps}{4}\right\}\,\Bigg\vert\,\Omega_\delta\right)<\frac{1}{2}.\end{align*}
Hence
\[
\P\left(\left\{\sup_{0\le t\le T}\|Y^x_t-u_t^x\|^2_H\le \frac{\eps}{4}\right\}\,\Bigg\vert\,\Omega_\delta\right)\ge\frac{1}{2},\]
uniformly for $x\in K$.
And thus, by $\delta<\sqrt{\eps}$,
\begin{align*}
&\P\left(\left\{\sup_{0\le t\le T}\|X^x_t-u_t^x\|^2_H\le\eps\right\}\right)\\
\geq&\P\left(\left\{\sup_{0\le t\le T}\|Y^x_t-u_t^x\|^2_H\le\frac{\eps}{4}\right\}\cap\left\{\sup_{0\le t\le T}\|B\,W_t\|^2_H\le\frac{\eps}{4}\right\}\right)\\
=&\P\left(\left\{\sup_{0\le t\le T}\|Y^x_t-u_t^x\|^2_H\le\frac{\eps}{4}\right\}\cap\Omega_{\sqrt{\eps}}\right)\\
\geq&\P\left(\left\{\sup_{0\le t\le T}\|Y^x_t-u_t^x\|^2_H\le\frac{\eps}{4}\right\}\cap\Omega_{\delta}\right)\ge\frac{\P(\Omega_\delta)}{2}>0,
\end{align*}
uniformly for $x\in K$.
The claim is proved.
\end{proof}

The next lemma verifies the second condition~\eqref{eq:KPS2} in the lower bound technique in 
Lemma~\ref{lem:KPS}. Note that we do not need to verify tightness, as the method proves a powerful alternative the the Krylov-Bogoliubov method~\cite{DPZ:96}.

\begin{lemma}\label{lem:Q1lem}
Assume that Hypotheses~\ref{hyp:hemicont},~\ref{hyp:coerc},~\ref{hyp:monotonicity},~\ref{hyp:growth},~and~\ref{hyp:regularity} are valid. Then for each $\eps>0$ and each bounded set $J\subset H$ there exists a constant $R(\eps,\|J\|_H)>0$ such that the sublevel set \[K:=\{\Theta(\cdot)\le R(\eps,\|J\|_H)\}\subset H\]
satisfies
\[\inf_{x\in J}\liminf_{T\to\infty} Q_T(x,K)>1-\eps.\]
Here, $\Theta$ is as in \eqref{eq:theta}.
\end{lemma}
\begin{proof}
Let $\eps>0$, $J\subset H$ be bounded and $x\in J$. For $R>0$, $K_R:=\{\Theta\le R\}$ is a measurable set.
By
\[\langle A(x),x\rangle\le -\Theta(x) \]
and by It\^o's formula, assuming Hypotheses~\ref{hyp:coerc} and~\ref{hyp:regularity}, there exists a constant $C>0$ such that
\[\frac{1}{T}\E\int_0^T\Theta(X_s^x)\,\ud s\le C\left(\|x\|_H^2+1\right),\quad\text{for}\;\; T\ge 1,\]
compare also with~\eqref{eq:apriori4}.
As a consequence,
\begin{align*}
Q_T(x,K_R)=&\frac{1}{T}\int_0^T P_s(x,K_R)\,\ud s
\ge\frac{1}{T}\int_0^T\left(1-\frac{\E[\Theta(X_s^x)]}{R}\right)\,\ud s\\
\ge&1-\frac{C}{R}\left(\|J\|_H^2+1\right).
\end{align*}
Choosing $R(\eps,\|J\|_H)>\eps^{-1} C(\|J\|^2_H+1)$ yields the claim with
\[K:=K_{R(\eps,\|J\|_H)}=\{x\in H\,\colon\,\Theta(x)\le R(\eps,\|J\|_H)\}.\]
\end{proof}

The next lemma verifies the first condition~\eqref{eq:KPS1} in the lower bound technique in Lemma~\ref{lem:KPS}.

\begin{lemma}\label{lem:Q2lem}
Assume Hypotheses~\ref{hyp:hemicont},~\ref{hyp:coerc},~\ref{hyp:coerc-g},~\ref{hyp:monotonicity},~\ref{hyp:monotonicity-g},~\ref{hyp:growth},~\ref{hyp:angle},~\ref{hyp:regularity}, and \ref{hyp:reg2}. Then for each $\delta>0$ and each bounded set $J\subset H$,
\[\inf_{x\in J}\liminf_{T\to\infty} Q_T(x,B_\delta(0))>0.\]
\end{lemma}
\begin{proof}
Let $\delta>0$, $J\subset H$ be bounded, $x\in J$ and $K=K_{R(\frac{1}{2},\|J\|_H)}$ be as in 
Lemma~\ref{lem:Q1lem}. By 
Lemma~\ref{lem:decay}, there exists $T_0>0$ corresponding to $K$ such that
\[\|u_{T_0}^z\|_H\le\frac{\delta}{2}\]
for all $z\in K$. Using 
Lemma~\ref{lem:coupling}, yields
\[P_{T_0}(z,B_\delta(0))=\P(\|X^z_{T_0}\|_H\le\delta)\ge\P\left(\|X^z_{T_0}-u^z_{T_0}\|_H\le\frac{\delta}{2}\right)\ge \gamma_1>0,\]
where $\gamma_1=\gamma_1(T_0,\delta)$ is independent of $z\in K$. Thus by 
Lemma~\ref{lem:Q1lem} with $\eps=\frac{1}{2}$,
\begin{align*}
&\liminf_{T\to\infty}Q_T(x,B_\delta(0))
=\liminf_{T\to\infty}\frac{1}{T}\int_0^T P_s(x,B_\delta(0))\,\ud s\\
=&\liminf_{T\to\infty}\frac{1}{T}\int_0^T P_{s+T_0}(x,B_\delta(0))\,\ud s
=\liminf_{T\to\infty}\frac{1}{T}\int_0^T\int_H P_s(x,\ud z)P_{T_0}(z,B_{\delta}(0))\,\ud s\\
\ge&\liminf_{T\to\infty}\frac{1}{T}\int_0^T\int_K P_s(x,\ud z)P_{T_0}(z,B_{\delta}(0))\,\ud s
\ge\liminf_{T\to\infty}\gamma_1\frac{1}{T}\int_0^T\int_K P_s(x,\ud z)\,\ud s\\
\ge&\gamma_1\liminf_{T\to\infty}Q_T(x,K)\ge\frac{\gamma_1}{2}>0,
\end{align*}
where $\gamma_1=\gamma_1(T_0,\delta)=\gamma_1(\|J\|_H,\delta)$.
\end{proof}

Note that for proving the previous lemma, the estimate~\eqref{eq:apriori} would not be sufficient, as we cannot guarantee that small sublevels have lower bounded transition probabilities uniformly on bounded sets for large times. Therefore, we need to use the coupling with the deterministic flow, which has a good decay behavior by Lemma~\ref{lem:decay}.

We have proved the existence of a unique invariant measure $\mu_\ast$ for $(P_t)_{t\ge 0}$ and thus Theorem~\ref{thm:main} by an application of Lemma~\ref{lem:KPS}. The concentration property can be seen as follows:
\begin{lemma}[Concentration]\label{lem:concentration}
Assume that Hypotheses~\ref{hyp:hemicont},~\ref{hyp:coerc},~\ref{hyp:coerc-g},~\ref{hyp:monotonicity},~\ref{hyp:monotonicity-g}, \ref{hyp:growth}, and \ref{hyp:regularity} are valid.
Then any invariant probability measure $\mu_\ast$ has finite $H$-moments of order $\alpha+\beta$. In addition,
\begin{equation}\label{eq:momentbound}
\int_{H} \|z\|^{\alpha+\beta}_H\,\mu_\ast(\ud z)\le\frac{2(c_1+c_2)}{c_0^\alpha\lambda_0\delta_1(\beta+2)}.
\end{equation}
In particular, 
\begin{equation}\label{eq:momentone2}
\int_{H} \|z\|_H^2\,\mu_\ast(\ud z)\le \left(\frac{2(c_1+c_2)}{c_0^\alpha\lambda_0\delta_1(\beta+2)}\right)^{2/(\alpha+\beta)}.
\end{equation}
Furthermore, $\mu_\ast$ has finite $V$-moments of order $\alpha$, and, in particular,
\begin{equation}
\int_{H} \|z\|^{\alpha}_V\,\mu_\ast(\ud z)\le\frac{\|B\|^2_{L_2(U,H)}}{\delta_1}.
\end{equation}
Here, the constants are as in Hypotheses~\ref{hyp:coerc},~\ref{hyp:monotonicity},~\ref{hyp:growth},~and~\ref{hyp:quant}, respectively.
\end{lemma}
\begin{proof}
By~\eqref{eq:apriori3} and 
Fact~\ref{rem:embedding}, it follows for every $z\in H$, for every $t\geq 0$, and every $\ell\in \mathbb{N}$ that
 \begin{equation}\label{eq:trunc}
 \begin{split}
&\left(\E\left[\int_0^t\|X_s^z\|_H^{\alpha+\beta}\,\ud s\right]\right) \wedge \ell \\
&\le \left(\frac{2(c_1+c_2)}{c_0^\alpha\lambda_0\delta_1(\beta+2)}t\right) \wedge \ell+\left(\frac{2}{c_0^\alpha\lambda_0\delta_1(\beta+2)}\|z\|_H^{\beta+2}\right)\wedge \ell.
\end{split}
\end{equation}
Since $\mu_\ast$ is invariant, we have 
\begin{equation}
\begin{split}
\int_{H} \left(\|z\|^{\alpha+\beta}_H \wedge \ell\right)\, \mu_\ast(\ud z)=
\int_{H} \mathbb{E}\left[\|X^z_t\|^{\alpha+\beta}_H \wedge\ell \right]\,\mu_\ast(\ud z).
\end{split}
\end{equation}
With the help of Fubini's Theorem, the preceding equality yields for $t>0$,
\begin{equation}
\begin{split}
\int_{H} \left(\|z\|^{\alpha+\beta}_H \wedge \ell\right) \mu_\ast(\ud z)&=
\frac{1}{t}\int_{0}^{t} \int_{H} \left(\|z\|^{\alpha+\beta}_H \wedge \ell\right)\, \mu_\ast(\ud z)\,\ud s\\
&=
\frac{1}{t}\int_{0}^{t}\int_{H} 
\mathbb{E}\left[\|X^z_s\|^{\alpha+\beta}_H \wedge \ell \right]\,\mu_\ast(\ud z)\,\ud s \\
&= \int_{H} 
\mathbb{E}\left[\frac{1}{t}\int_{0}^{t}\,\left(\|X^z_s\|^{\alpha+\beta}_H \wedge \ell\right) \,\ud s \right]\, \mu_\ast(\ud z) \\
&\leq 
\int_{H} \left(\left(
\frac{1}{t}\int_{0}^{t} \mathbb{E}\left[\|X^z_s\|^{\alpha+\beta}_H\right]
\,\ud s\right)\wedge \ell \right)
\mu_\ast(\ud z).
\end{split}
\end{equation}
The preceding inequality, together with~\eqref{eq:trunc}, implies for $\alpha>2$, $t>0$,
\begin{align*}
&\int_{H} \left(\|z\|^{\alpha+\beta}_H \wedge \ell\right)\, \mu_\ast(\ud z)\\
\leq &\left(\frac{2(c_1+c_2)}{c_0^\alpha\lambda_0\delta_1(\beta+2)}\right)\wedge \ell+\frac{1}{t}\frac{2}{c_0^\alpha\lambda_0\delta_1(\beta+2)}\int_{H} \left(\|z\|^{\beta+2}_H \wedge \ell\right)\, \mu_\ast(\ud z)\\
\leq &\left(\frac{2(c_1+c_2)}{c_0^\alpha\lambda_0\delta_1(\beta+2)}\right)\wedge \ell+\frac{1}{t}\frac{2}{c_0^\alpha\lambda_0\delta_1(\beta+2)}\int_{H}\left( \left(\|z\|^{\alpha+\beta}_H+1\right) \wedge \ell\right)\, \mu_\ast(\ud z),
\end{align*}
and thus for $t\ge\frac{c_0^\alpha\lambda_0\delta_1(\beta+2)}{2}$,
\begin{align*}
&\left(1-\frac{1}{t}\frac{2}{c_0^\alpha\lambda_0\delta_1(\beta+2)}\right)\int_{H} \left(\|z\|^{\alpha+\beta}_H \wedge \ell\right)\, \mu_\ast(\ud z)\\
\leq &\left(\frac{2(c_1+c_2)}{c_0^\alpha\lambda_0\delta_1(\beta+2)}\right)\wedge \ell+\frac{1}{t}\frac{2}{c_0^\alpha\lambda_0\delta_1(\beta+2)}(1\wedge\ell).
\end{align*}
For any $\eps>0$, we can find a large enough $t>0$, such that
\begin{equation}
\int_{H} \left(\|z\|^{\alpha+\beta}_H \wedge\ell\right) \mu_\ast(\ud z)\le\left(\frac{2(c_1+c_2)}{c_0^\alpha\lambda_0\delta_1(\beta+2)}\right)\wedge \ell+\eps.
\end{equation}
After sending $\ell \to \infty$, we obtain with the help of Fatou's lemma that
\begin{equation}
\int_{H} \|z\|^{\alpha+\beta}_H  \mu_\ast(\ud z)\le\frac{2(c_1+c_2)}{c_0^\alpha\lambda_0\delta_1(\beta+2)}+\eps.
\end{equation}
Inequality~\eqref{eq:momentbound} follows from the fact that $\eps>0$ was arbitrary. Inequality~\eqref{eq:momentone2} follows directly by Jensen's inequality.

Furthermore, \eqref{eq:apriori4} implies for $\ell\in\N$, $z\in H$, $t\ge 0$,
\begin{equation}\label{eq:trunc2}
 \begin{split}
\left(\E\left[\int_0^t\|X_s^z\|_V^{\alpha}\,\ud s\right]\right) \wedge \ell &\le \left(\frac{\|B\|^2_{L_2(U,H)}}{\delta_1}\right)t \wedge \ell+\left(\frac{1}{\delta_1}\|z\|_H^{2}\right)\wedge \ell.
\end{split}
\end{equation}
and hence the proof can be completed by repeating the argument above.
\end{proof}

\section{Mixing times}\label{sec:mixing}

In this section, we rigorously define the concept of mixing times for uniquely ergodic systems in the Wasserstein distance of order two. 

We start recalling the definition of coupling between two probability measures. Let $H$ be a Hilbert space equipped with the inner product $\<\cdot,\cdot\>_{H}$ and its induced norm $\|\cdot\|_{H}$.
Let $\mathcal{P}$ be the set of probability measures defined in the measurable space $(H,\mathcal{B}(H))$, where $\mathcal{B}(H)$ is the Borel $\sigma$-algebra of $H$.

Given $\mu_1\in \mathcal{P}$ and $\mu_2\in \mathcal{P}$,
we say that a probability measure $\Pi$ defined in the product space $(H\times H,\mathcal{B}(H)\otimes \mathcal{B}(H))$ is a coupling between $\mu_1$ and $\mu_2$ if and only if
\[
\Pi(A\times H)=\mu_1(A) \quad \textrm{ and } \quad
\Pi(H\times A)=\mu_2(A)\quad \textrm{ for any measurable set }\quad  A\in \mathcal{B}(H).
\]
We then denote by $\mathcal{C}(\mu_1,\mu_2)$ the set of all couplings between $\mu_1$ and $\mu_2$. Clearly, $\mathcal{C}(\mu_1,\mu_2)\not=\emptyset$. 

Let $\mathcal{P}_2$ be the set of Borel probability measures on $(H,\mathcal{B}(H))$ with 
 finite second moment, that is, 
 $\mathcal{P}_2\subset \mathcal{P}$ and
\begin{equation}
 \int_{H} \|x\|^2_H \mu(\ud x)<\infty\quad \textrm{ for all }\quad \mu\in \mathcal{P}_2.
\end{equation}
The Wasserstein distance of order two,
$\mathcal{W}_2:\mathcal{P}_2\times \mathcal{P}_2\to [0,\infty)$, is defined as
\begin{equation}
\begin{split}
\mathcal{W}_2(\mu_1,\mu_2):=\left(\inf_{\Pi\in \mathcal{C}(\mu_1,\mu_2)}
\int_{H\times H} \|x-y\|^2_H\Pi(\ud x,\ud y)\right)^{1/2}
\end{split}
\end{equation}
for any $\mu_1,\mu_2\in \mathcal{P}_2$.
By \cite[Theorem 6.9]{villani} we have that $(\mathcal{P}_2,\mathcal{W}_2)$ is a Polish metric space that metrizes the weak topology on $\mathcal{P}_2$.

By Theorem~\ref{thm:main},
we have that the random dynamics  \eqref{eq:model} are uniquely ergodic. In other words, there exists a unique probability measure $\mu_\ast\in \mathcal{P}$ such that 
for any initial condition $x\in H$ it satisfies the marginal
$X^x_t$ converges in distribution to $\mu_\ast$ as $t\to \infty$.
In addition, if $\mu_\ast\in \mathcal{P}_2$ then $\mathcal{W}_2(P^*_t\delta_x,\mu_\ast)\to 0$ as $t\to \infty$. 
Since the process $(X^x_t)_{t\geq 0}$ is Markovian by \cite[Proposition 4.3.5]{LR:15}, the map $[0,\infty)\ni t\mapsto  \mathcal{W}_2(P^*_t\delta_x,\mu_\ast)\in [0,\infty)$ 
is non-increasing, see \cite[Lemma~B.3 (Monotonicity)]{BCL:22}.

Let us recall the definition of an $\e$-mixing time, see for instance~\cite{DIA:96,LPW:09} in the context of Markov chains.
Given a prescribed error $\e>0$ we define the $\e$-mixing time (with respect to $\mathcal{W}_2$) as 
\begin{equation}\label{eq:mixing}
\tau^x_{\textsf{mix}}(\e):=\inf\{t\geq 0: \mathcal{W}_2(P^*_t\delta_x,\mu_\ast)\leq \e\}.   
\end{equation}

\begin{lemma}[Disintegration]\label{lem:disi}
Assume hypotheses and notation of Theorem~\ref{thm:main}.
For any $x\in H$ and $0<s\leq t$
it follows that
\begin{equation}\label{eq:uno}
\left(\mathcal{W}_2(P^*_t\delta_x,P^*_t\mu_\ast)\right)^2\leq \int_{H}
\int_{H}\left(\mathcal{W}_2(P^*_s\delta_z,P^*_s\delta_y)\right)^2 P^*_{t-s}\delta_x(\ud z)\mu_\ast(\ud y).
\end{equation}
In particular, it follows that
\begin{equation}\label{eq:dos}
\left(\mathcal{W}_2(P^*_t\delta_x,\mu_\ast)\right)^2=\left(\mathcal{W}_2(P^*_t\delta_x,P^*_t\mu_\ast)\right)^2
\leq \int_{H} \left(\mathcal{W}_2(P^*_t\delta_x,P^*_t\delta_y)\right)^2\mu_\ast(\ud y).
\end{equation}
\end{lemma}
\begin{proof}
Recall Kantorovich's duality, Theorem~5.10 in \cite{villani},
for $\mathcal{W}_2$, that is,
\[
\left(\mathcal{W}_2(P^*_t\delta_x,P^*_t\mu_\ast)\right)^2=\sup_{(f,g)}\left(\int_{H}f(z)P^*_t\delta_x(\ud z)-\int_{H}g(z)P^*_t\mu_\ast(\ud z)\right),
\]
where the supremum is running over all continuous and bounded functions $f$ and $g$ such that
$f(x)-g(y)\leq \|x-y\|^2$ for all $x,y\in H$.
Since the family of processes $((X^{x}_t)_{t\geq 0})_{x\in H}$ is Markovian, for any $0< s\leq t$ we have
\begin{equation}
\int_{H}f(z)P^*_t\delta_x(\ud z)=
\int_{H} \mathbb{E}[f(X^z_s)]P^*_{t-s}\delta_x(\ud z)=\int_{H}
\int_{H} \mathbb{E}[f(X^z_s)]P^*_{t-s}\delta_x(\ud z)P^*_{t-s}\mu_\ast(\ud y),
\end{equation}
and
\begin{equation}
\int_{H}g(z)P^*_t\mu_\ast(\ud z)=\int_{H} 
\mathbb{E}[g(X^y_s)]P^*_{t-s}\mu_\ast(\ud y)=\int_{H}
\int_{H} \mathbb{E}[g(X^y_s)]P^*_{t-s}\mu_\ast(\ud y)P^*_{t-s}\delta_x(\ud z).
\end{equation}
We then observe that
\begin{equation}\label{eq:couplingtwo}
\begin{split}
&\int_{H}f(z)P^*_t\delta_x(\ud z)-\int_{H}g(z)P^*_t\mu_\ast(\ud z)\\
&=
\int_{H}
\int_{H} \mathbb{E}[f(X^z_s)]P^*_{t-s}\delta_x(\ud z)P^*_{t-s}\mu_\ast(\ud y)-\int_{H}
\int_{H} \mathbb{E}[g(X^y_s)]P^*_{t-s}\mu_\ast(\ud y)P^*_{t-s}\delta_x(\ud z)\\
&=
\int_{H}
\int_{H} \mathbb{E}\left[f(X^z_s)-g(X^y_s)\right]P^*_{t-s}\delta_x(\ud z)P^*_{t-s}\mu_\ast(\ud y)
\end{split}
\end{equation}
Now, taking the supremum over all continuous and bounded functions $f$ and $g$ such that
$f(x)-g(y)\leq \|x-y\|^2$ for all $x,y\in H$ in both sides of \eqref{eq:couplingtwo} we obtain
\begin{equation}
\left(\mathcal{W}_2(P^*_t\delta_x,P^*_t\mu_\ast)\right)^2\leq \int_{H}
\int_{H}\left(\mathcal{W}_2(P^*_s\delta_z,P^*_s\delta_y)\right)^2 P^*_{t-s}\delta_x(\ud z)P^*_{t-s}\mu_\ast(\ud y)
\end{equation}
for all $0<s\leq t$, which concludes the proof of \eqref{eq:uno}.
Since $P^*_{t-s}\mu_\ast=\mu_\ast$ for all $0\leq s\leq t$ and $P^*_{0}\delta_x=\delta_x$, we deduce
\begin{equation}
\left(\mathcal{W}_2(P^*_t\delta_x,P^*_t\mu_\ast)\right)^2\leq 
\int_{H}\left(\mathcal{W}_2(P^*_t\delta_x,P^*_t\delta_y)\right)^2\mu_\ast(\ud y)
\end{equation}
yielding the proof of \eqref{eq:dos}.
\end{proof}

\subsection{Proof of Theorem \ref{thm:main2}}\label{sec:proof2}

The following theorem provides a quantitative upper bound for the $\e$-mixing time.

\begin{theorem}[Mixing time]\label{thm:mixing}
Assume Hypotheses \ref{hyp:hemicont},  \ref{hyp:coerc}, \ref{hyp:coerc-g}, \ref{hyp:monotonicity}, \ref{hyp:monotonicity-g}, \ref{hyp:growth} and \ref{hyp:regularity}, \ref{hyp:reg2}, \ref{hyp:quant}, and \ref{hyp:mixing} hold true. Suppose that $\eta\equiv 0$ or $\rho\equiv 0$. Let $\gamma\in (0,\delta_2]$ be as in Hypothesis~\ref{hyp:mixing}.
Then for any $x\in H$ and $\e>0$ it follows that the $\e$-mixing time $\tau^x_{\textup{\textsf{mix}}}(\e)$ satisfies the following non-asymptotic estimate:
\begin{equation}\label{ec:mix}
\tau^x_{\textup{\textsf{mix}}}(\e)\leq 
\frac{2}{\gamma}\left[\frac{C_2}{\lambda_0\delta_1(\beta+2)}\|x\|^{\beta+2}_H+\log\left(\|x\|_H+\left(\frac{2(c_1+c_2)}{c_0^\alpha\lambda_0\delta_1(\beta+2)}\right)^{1/(\alpha+\beta)}\right)+\log\left(\frac{1}{\e}\right)\right].
\end{equation}
\end{theorem}
\begin{proof}
Note that by Hypothesis \ref{hyp:mixing},
by Lemma~\ref{lem:disi}, and similar ideas as in the proofs of Lemma~\ref{lem:apriori}, possibly after passing to the limit in the Galerkin approximation, and Proposition~\ref{prop:e}, as well as, \eqref{eq:momentone2} in Lemma~\ref{lem:concentration},
we have by Hypothesis \ref{hyp:mixing}
\begin{equation}\label{eq:Wmix}
\begin{split}
&\mathcal{W}_2(P^*_t\delta_x,\mu_\ast)\\
\leq& \left(\int_{H} \left(\mathcal{W}_2(P^*_t\delta_x,P^*_t\delta_z)\right)^2\mu_\ast(\ud z)\right)^{1/2}\\
\leq& \left(\int_{H} \E\|X_t^x-X_t^z\|_H^2\,\mu_\ast(\ud z)\right)^{1/2}\\
\leq&  \exp\left(-\frac{\delta_2}{2} t+\frac{C_2}{\lambda_0\delta_1(\beta+2)}\left(\|x\|^{\beta+2}_H+(c_1+c_2+c_3)t\right)\right) \left(\int_{H} 
\|z-x\|_H^2 \,\mu_\ast(\ud z)\right)^{1/2}\\ 
\leq&  \exp\left(-\frac{\gamma}{2} t+\frac{C_2}{\lambda_0\delta_1(\beta+2)}\|x\|^{\beta+2}_H\right) \left(\int_{H} 
\|z-x\|_H^2\, \mu_\ast(\ud z)\right)^{1/2}\\ 
\leq&  \exp\left(-\frac{\gamma}{2} t+\frac{C_2}{\lambda_0\delta_1(\beta+2)}\|x\|^{\beta+2}_H\right)\|x\|_H\\
&+
 \exp\left(-\frac{\gamma}{2} t+\frac{C_2}{\lambda_0\delta_1(\beta+2)}\|x\|^{\beta+2}_H\right)
\left(\int_{H} 
\|z\|_H^2 \,\mu_\ast(\ud z)\right)^{1/2}\\
\leq&  
 \exp\left(-\frac{\gamma}{2} t+\frac{C_2}{\lambda_0\delta_1(\beta+2)}\|x\|^{\beta+2}_H\right)\left(\|x\|_H+\left(\frac{2(c_1+c_2)}{c_0^\alpha\lambda_0\delta_1(\beta+2)}\right)^{1/(\alpha+\beta)}\right)
\\
=:&h(t),
\end{split}
\end{equation}
for any $x\in H$ and $t\geq 0$.
Note that the function $h$ is strictly decreasing. Furthermore, note that the exponential term does not depend on $\|z\|^{\beta+2}$ by a modification of the proof of Proposition \ref{prop:e}, as we are in the locally monotone case.

Let $\e>0$ be fixed and let $t_\e>0$ be such that $h(t_\e)=\e$, that is, 
\begin{equation}
t_\e:=\frac{2}{\gamma}\left[\frac{C_2}{\lambda_0\delta_1(\beta+2)}\|x\|^{\beta+2}_H+\log\left(\|x\|_H+\left(\frac{2(c_1+c_2)}{c_0^\alpha\lambda_0\delta_1(\beta+2)}\right)^{1/(\alpha+\beta)}\right)+\log\left(\frac{1}{\e}\right)\right]
\end{equation}
By the definition of mixing time \eqref{eq:mixing} we obtain $\tau^x_{\textsf{mix}}(\e)\leq t_\e$ and hence \eqref{ec:mix}. The proof is complete.
\end{proof}

Thus we have proved Theorem \ref{thm:main2}. As a result we obtain the exponential ergodicity of the semigroup $(P_t)_{t\ge 0}$.

\section*{Declarations}
\noindent
\textbf{Acknowledgments.} 
GB would like to express his gratitude to University of Helsinki, Aalto University School of Science, and Instituto Superior T\'ecnico for all the facilities used along the realization of this work. JMT would like to thank the Department of Mathematics and Systems Analysis, Aalto University for providing its facilities and free coffee for this research.
 JMT would like to thank Dirk Bl\"omker (Universit\"at Augsburg, Germany) and Martin Hairer (Imperial College, UK and EPFL, Switzerland) for fruitful discussions and useful comments on stability of the 2D stochastic Navier-Stokes equations. JMT would like to thank Erika Hausenblas (Montanuniversit\"at Leoben, Austria) and Mark Veraar (Delft University of Technology, Netherlands) for comments on a previous version of this work.

\noindent
\textbf{Funding.}
The research of GB has been supported by the Finnish Centre of Excellence in Randomness and Structures (decision numbers 346306 and 346308). The research of both authors was partially supported by the European Union's Horizon Europe research and innovation programme under the Marie Sk\l{}odowska-Curie Actions Staff Exchanges (Grant agreement no.~101183168 -- LiBERA, Call: HORIZON-MSCA-2023-SE-01). Also, the research of Gerardo Barrera is partially funded by Funda\c{c}\~ao para a Ci\^encia e Tecnologia (FCT), Portugal, through grant No. UID/4459/2025.

\noindent
\textbf{Disclaimer.}
Co-funded by the European Union. Views and opinions expressed are however those of the authors only and do not necessarily reflect those of
the European Union or the European Education and Culture Executive Agency (EACEA). Neither the European Union nor EACEA can be held responsible for them.

\noindent
\textbf{Ethical approval.} Not applicable.

\noindent
\textbf{Competing interests.} The authors declare that they have no conflict of interest.

\noindent
\textbf{Authors' contributions.}
All authors have contributed equally to the paper.

\noindent
\textbf{Availability of data and materials.} Data sharing not applicable to this article as no data-sets were generated or analyzed during the current study.

\end{document}